\documentclass[journal]{IEEEtran}

\usepackage{graphicx}
\graphicspath{{../pdf/}{../jpeg/}}
\DeclareGraphicsExtensions{.pdf,.jpeg,.png}

\usepackage[cmex10]{amsmath}
\usepackage{array}
\usepackage{mdwmath}
\usepackage{mdwtab}
\usepackage{eqparbox}
\usepackage{url}
\usepackage{algorithm} 
\usepackage{algorithmic} 
\usepackage{amsmath}
\usepackage{graphics}
\usepackage{epsfig}
\usepackage{multirow}
\usepackage{setspace}
\usepackage{booktabs}
\usepackage{mathrsfs}
\usepackage{epsfig}
\usepackage{color}
\usepackage{amssymb,latexsym,amsfonts,amsmath}
\usepackage{graphicx}
\usepackage{cite}
\usepackage{subfigure,graphicx}
\usepackage{tabularx}

\allowdisplaybreaks[4]
\usepackage[
linkcolor=red,citecolor=blue,urlcolor=green,]{hyperref}
\usepackage[doipre={doi:~}]{uri}

\newtheorem{theorem}{Theorem}
\newtheorem{lemma}{Lemma}

\newtheorem{corollary}{Corollary}
\newtheorem{definition}{Definition}
\newtheorem{assumption}{Assumption}
\newtheorem{remark}{Remark}

\newcommand{\R}{{\mathbb{R}}}

\begin{document}
    \title{A Network Transformation Mapping Approach to  Synchronization of Multi-Agent Systems With Disconnected Switching Topologies}
  \author{Haotian~Xu,
          Bohui~Wang,~\IEEEmembership{Senior Member,~IEEE},
          Shuai~Liu,~\IEEEmembership{Member,~IEEE},
          Chao~Shen,~\IEEEmembership{Senior Member,~IEEE},
          Xiangyu Meng,~\IEEEmembership{Senior Member,~IEEE},
          Guanghui Wen,~\IEEEmembership{Senior Member,~IEEE}


\thanks{H. Xu and S. Liu are with the School of Control Science and	Engineering, Shandong University, Jinan 250061, China~(e-mail: xuhaotian\_1993@126.com; liushuai@sdu.edu.cn)}
\thanks{B. Wang and C. Shen are with the School of Cyber Science and Engineering, Xi'an Jiaotong University, Xi'an 710049, China~(e-mail: wang31aa@126.com; cshen@sei.xjtu.edu.cn)}
\thanks{X. Meng is with the Division of Electrical and Computer Engineering, Louisiana State University, Baton Rouge, LA 70803, USA. (E-mail:
xmeng5@lsu.edu}
}

\markboth{IEEE Transactions on Automatic Control
}{Roberg \MakeLowercase{\textit{et al.}}: High-Efficiency Diode and Transistor Rectifiers}

\maketitle

\begin{abstract}
This paper focuses on the multi-agent synchronization problem with an open-loop unstable leader and followers under the switching topologies. For this issue, the typical approach is intermittent communication (including a spanning tree intermittently) or fast switching strategy. {\color{blue}We here consider a more general scenario where each communication link between two agents is randomly disconnected and reconnected, and the durations of both the disconnected intervals and the connected intervals follow negative exponential distributions.} To handle this issue, we propose a network transformation mapping method that divides the communication network into reachable and unreachable parts at any time. A node can access the leader's information and synchronize only when it lies in the reachable part; otherwise, it cannot. For each node, the synchronization speed is designed such that its convergence amplitude in the reachable part exceeds its divergence amplitude in the unreachable part. Hence, all nodes achieve asymptotic synchronization over the entire time domain. We further develop adaptive strategies for coupling gains to reduce the computational complexity introduced by the network transformation mapping. The proposed method is also applicable to jointly connected switching topologies and distributed observers under topologies that lack a spanning tree at any time. Finally, two numerical simulations—synchronous control of multi-one-link manipulator systems and multi-motor systems based on Fieldbus—are provided to demonstrate the effectiveness of our approach.
\end{abstract}


%
\IEEEpeerreviewmaketitle


\section{Introduction}

\IEEEPARstart{M}{ulti}-agent systems is an important theoretical tool to solve future intelligent control and applications, such as cyber-physics systems \cite{wang2023timevary,10654516}, Internet of vehicles \cite{10185611,10012572}, Wireless sensor networks \cite{ZHENG2024111839,Xu2024JAS}, and so on. Up to now, there has been extensive research on multi-agent systems under a static network including the leader-following consensus problem \cite{2014The}, cooperative output regulation problem \cite{8809186}, and coordinated mission rendezvous problem \cite{WANG2021109931}, as well as cross integration with other research topics, such as network security \cite{10504568,9760723}, reinforcement learning \cite{9597491,10452844}, distributed optimization \cite{10664010,9547726}, and so on. Moreover, in recent 20 years, the multi-agent problem under switching topologies has also been thoroughly studied by \cite{Wei2010Leader,Cheng2020Seeking,9102431Exponential,Li2021DynamicConsensus,Huang2017The,BohuiCooperative,Wen2017Distributed,2017Pinning,Zhu2016A,khan2019output,2020Comments} and references therein. 

The existing literature on multi-agent systems under switching topologies primarily investigates two categories of problems. The first category focuses on scenarios where the leader or virtual reference system (hereafter collectively referred to as the reference system) is bounded. In such cases, multi-agent systems can readily achieve reference tracking under jointly connected switching topologies \cite{zhang2023DSESWTICH,Lu2020Robust,10323478,khan2019output,liu2024distributed,ZHANG2024111755}. The second category addresses situations involving unbounded reference systems, where current studies predominantly employ either fast switching methods \cite{HE2021110021,StilwellSufficient,HE2021110021,wang2024distributed} or intermittent communication strategies \cite{10328694,ZUO2023110952,BohuiCooperative,2017Pinning} to achieve consensus under switching topologies. 

{\color{blue}Notably, existing approaches for unbounded reference systems remain inapplicable to scenarios of real-time disconnected switching networks (real-time disconnected means that at any time, there does not exist a spanning tree that allows all followers to access the leader's information).} In practical implementations, such switching topologies are prevalent in large-scale multi-agent systems due to persistent challenges including bandwidth constraints, communication network failures, and cyber attacks. Moreover, multi-agent systems with unbounded reference systems represent a significant class of applications, particularly in scenarios such as second-order or high-order integrator systems, multi-motor synchronous control systems based on Fieldbus, and formation problems about UAVs, unmanned submarines, spacecraft, and so on \cite{9681160,9484806,7122276}. Therefore, filling this critical research gap in the failure of multi-agent systems to track unbounded reference systems under real-time disconnected switching networks holds substantial theoretical and practical significance.

However, addressing this issue is really difficult because there are some communication links in the real-time disconnect switching networks that have negative effects, leading to the inability of multi-agent systems with unbounded reference systems to achieve consistency. For example, when agent $i$ gains access to the reference system's information during some time interval, its tracking error dynamics $e_i(t)$ should theoretically exhibit convergence tendency under appropriate control laws---a necessary condition for achieving consensus in switching topologies. However, if during the same interval agent $i$ simultaneously receives state information from neighboring agent $j$ that lacks reference system accessibility, {\color{blue}the divergent error dynamics $e_j(t)$ (note that $e_j(t)$ diverges when $j$ has no access to the leader during this interval; this does not mean $e_j(t)$ always diverges, as it may converge when $j$ has access) will be transmitted through the interconnection.} This error contamination phenomenon forces the $e_i(t)$ that should exhibit a convergent trend to inherit the divergence behavior. Unlike bounded reference systems where errors stay within invariant sets and contamination is prevented, there exist divergent interference signals that are continuously generated by the unbounded reference trajectories. Consequently, there always exists a subset of agents whose error dynamics diverge persistently over every switching interval, ultimately violating the consensus conditions derived from classical connectivity assumptions.

The content analyzed aforementioned is the fundamental reason why existing research cannot achieve synchronization between multi-agent systems and an unbounded reference system under real-time disconnected switching topologies. To overcome this challenge, this paper innovatively proposes a network transformation mapping method that enables agents to automatically determine whether received information is valid, thereby fundamentally solving the error contamination issue that prevents synchronization of multi-agent systems with unbounded reference systems under real-time disconnected switching topologies. After resolving this most critical issue, several key challenges still need to be overcome to achieve consensus in such systems. 1) Effective Time Estimation: When communication networks experience probabilistic interruptions and reconnections, estimating the effective duration during which each agent acquires reference system information becomes crucial for deriving sufficient consensus conditions. 2) Distribution Maintenance: While employing network transformation mapping, ensuring agents avoid using centralized information remains another essential challenge that must be addressed. 3) An adaptive coupling gain strategy is proposed to circumvent the computational complexity inherent in the network transformation mapping, eliminating the need for explicit precomputation of thresholds or topology parameters, thereby enabling synchronization under real-time disconnected switching topologies.

The main contributions of this paper lie in three aspects. 1) We propose a network transformation mapping method whereby agents autonomously determine the utility of neighbor information through localized exchange of network transformation mapping parameters. This approach prevents error propagation from divergent agents to those exhibiting convergent dynamics. 2) This paper establishes, for the first time, synchronization of unbounded multi-agent systems under real-time disconnected switching topologies. Existing approaches fundamentally require fast switching strategies to achieve consensus in such networks \cite{StilwellSufficient,HE2021110021}. 3) The proposed network transformation mapping method is also proven in the case of synchronization for multi-agent systems with jointly connected switching topologies and distributed observers against switching topologies without a spanning tree at any time.

The paper is organized as follows. Section \ref{sec2} introduces notations and formulates problems. A new synchronous controller, as well as the developed transformation network, are proposed in Section \ref{sec3}. And the synchronization with adaptive coupling gain is proved in Section \ref{sec4}; Then, Section \ref{sec5} proposes the further application of network transformation mapping. Section \ref{sec6} simulates the method for verifying the effectiveness. Finally, the conclusions of this paper are given in Section \ref{sec7}.

\section{Preliminaries and problem formulation}\label{sec2}
\subsection{Preliminaries}
The notations and graph theory used in this paper are illustrated as follows.

\textbf{(N1)} Let $\mathbb{R}^n$ be the Euclidean space concerning $n$-dimentional column vertors and $\mathbb{R}^{m\times n}$ be the Euclidean space of $m\times n$-dimentional matrices. Denote $A^T$ the transpose of matrix $A$, and $sym\{A\}$ is defined by $A+A^T$. For the symmetric matrix $A$, $\bar{\lambda}(A)$ and $\underline{\lambda}(A)$ stands for  maximum and minimum eigenvalue of $A$ respectively. We cast $col\{A_1,\ldots,A_n\}$ as $[A_1^T,\ldots,A_n^T]^T$, and $diag\{A_1,\ldots,A_n\}$ as a block diagonal matrix with $A_i$ on its diagonal, where $A_1,\ldots,A_n$ are matrices with arbitrary dimentions. $1_n$ is the column vector belonging to $\mathbb{R}^n$ with all entries equal $1$, and $I_n$ is the identity matrix belonging to $\mathbb{R}^{n\times n}$.

\textbf{(N2)} A graph given by $\mathcal{G}=\{\mathcal{V},\mathcal{E}\}$ includes $N$ nodes (represented by $\mathcal{V}$) and arcs (represented by $\mathcal{E}$) among nodes. $\mathcal{A}\in\mathbb{R}^{N\times N}$ stands for the adjacency matrix of $\mathcal{G}$, whose element $\alpha_{ij}=1$ indicates an arc from $j$ to $i$ and $\alpha_{ij}=0$ otherwise. The Laplacian matrix associated with $\mathcal{G}$ is notated as $\mathcal{L}=\mathcal{D}-\mathcal{A}$ where $d_i=\sum_{j\neq i}\alpha_{ij}$ and $\mathcal{D}=diag\{d_1,\ldots,d_N\}$. A path from node $i$ to node $j$ is a sequence of arcs $(i,k_1)$, $(k_1,k_2)$, $\cdots$, $(k_l,j)$ with distinct nodes $i, k_1, \cdots, k_l, j$. A graph $\mathcal{G}$ contains a spanning tree rooted at $i$ if there is a path pointing from $i$ to any other nodes. Let $\sigma(t): t\longmapsto \mathcal{P}$ be a piecewise {\color{blue}constant} function, where the element in $\mathcal{P}$ is the index of networks. Then, we define time-dependent symbols such as $\mathcal{G}^\sigma$,  $\mathcal{E}^\sigma$, $\alpha_{ij}^\sigma$, and so on. Also, for a variable $s$, a symbol $s^\sigma$ means that it is a time-dependent variable and its value depends on $\sigma(t)$.

\textbf{(N3)} In this paper, $\|\cdot\|$ denotes 2-norm of a vector or a matrix. $\|x\|_P$ stands for $x^TPx$ with a given symmetric positive definite matrix $P$ and a vector $x$. $\otimes$ is the Kronecker product between matrices $A$ and $B$. $\wedge$ means logic AND, i.e., $a\wedge b=\min\{a,b\}$ for $a,b\in\mathbb{R}$. Correspondingly, $\vee$ is logic OR with definition $a\vee b=\max\{a,b\}$. In addition, large operator $\bigwedge$ and $\bigvee$ denote the multivariate logic operation. For example, $\bigwedge_{i=1}^na_i=a_1\wedge\cdots\wedge a_n$. 
where $A$ is a matrix.

Finally, due to the numerous symbols used in this article, to avoid confusion for readers regarding symbol definitions, we provide a list of key symbol mappings for their reference (see Table \ref{tab:notation}).
\begin{table}[htbp]
	\centering
	\caption{Notation Table}
	\label{tab:notation}
	\begin{tabularx}{\linewidth}{c>{\raggedright\arraybackslash}X}
		\hline
		Symbol & Meaning \\
		\hline
		$\kappa_i$ & Fixed coupling gain of agent $i$ \\
		$\gamma_i$ & Integral part of adaptive coupling gain \\
		$\omega_i$ & Instantaneous part of adaptive coupling gain\\
		$\kappa_*$ & Constant used to design feedback gain $K$\\
		$\theta_i^\sigma$ & $i$th diagonal element of diagonal matrix $\Theta_\star^\sigma$ \\
		$\Theta_\star^\sigma$ & Positive definite diagonal matrix satisfying $\operatorname{sym}\{\Theta_\star^\sigma \mathcal{H}_\star^\sigma\} > \kappa_* I_{\pi(\sigma)}$, where $\pi(\sigma)=|\mathcal{V}_\star^\sigma|$ \\
		$\mathcal{V}_\star^\sigma, \mathcal{V}_\diamond^\sigma$ & Reachable node set ($\star$) and unreachable node set ($\diamond$) after transformation \\
		$\mathcal{L}_\star^\sigma,\mathcal{L}_\diamond^\sigma$ & Laplacian matrix of graph $\mathcal{G}_\star^\sigma$ and $\mathcal{G}_\diamond^\sigma$ \\
		$\mathcal{H}_\star^\sigma$ & Augmented Laplacian of transformed reachable subgraph: $\mathcal{H}_\star^\sigma = \mathcal{L}_\star^\sigma + \mathcal{S}_\star^\sigma$ \\
		$B$ & Input matrix in agent dynamics: $\dot{x}_i = A x_i + f(x_i) + B u_i$ \\
		$\mathcal{B}^\sigma$ & Transformed adjacency weight matrix with entries $\beta_{ij}^\sigma$ \\
		$\beta_{ij}^\sigma$ & Transformed adjacency weight: $\beta_{ij}^\sigma = \alpha_{ij}^\sigma \delta_{ij}^\sigma$ \\
		$\alpha_{ij}^\sigma$ & Original adjacency weight of communication network ($0$ or $1$) \\
		$\delta_{ij}^\sigma$ & Logical operation result: $\delta_{ij}^\sigma = \ell_{i}^\sigma \wedge \ell_{j}^\sigma$ \\
		$\ell_{ik}^\sigma$ & Reachability indicator of node $i$: $1$ if there exists a directed path from a pinning node to $i$ \\
		$\sigma(t)$ & Switching signal, piecewise constant function taking values in finite set $\mathcal{P}$ \\
		$m_c$ & Long-run proportion of time that all arcs on a path are simultaneously up: $(\mu/(\lambda+\mu))^N$ \\
		$\wp_\star$ & Convergence rate of reachable subgraph (defined in Theorem 1) \\
		$\wp_\diamond$ & Divergence rate of unreachable subgraph (defined in Theorem 2) \\
		\hline
	\end{tabularx}
\end{table}

\subsection{Problem formulation}

Now, we formulate the problem of synchronization for nonlinear multi-agent systems with switching networks. The dynamics of the $i$th agent is assumed to be governed by
\begin{align}\label{sys}
\dot{x}_i=Ax_i+f(x_i)+Bu_i,~i=1,\ldots,N,
\end{align}
where $x_i\in\R^n$ and $u_i\in\R^m$ are, respectively, system states and control inputs; $A\in\R^{n\times n}$ and $B\in\R^{n\times m}$ are the system and input matrices with advisable dimension, respectively; $f(\cdot)$ represents the $n$-dimensional smooth vector field with respect to $x_i$. The multi-agent systems contain $N$ agents which exchange the information among a time-varying directed communication network $\mathcal{G}^\sigma$. The overall goal of this paper is to realize the state tracking of each agent to the reference signal $x_0$ under the switching topologies, i.e., achieve $\lim_{t\to\infty}\|x_i(t)-x_0(t)\|=0$ with
\begin{align}
\dot{x}_0(t)=Ax_0+f(x_0).
\end{align}
The trajectories of $x_0(t)$ are assumed to be unbounded.

It is assumed that only a part of agents can use the state information of the reference signal. Such nodes are the so-called pinning nodes. For saving of convenience, we define $\iota_i^\sigma=1$ if $i$ is a pinning node at time $t$. By regarding the reference signal as a virtual leader, all the pinning nodes can be cast as neighbors of the virtual leader. Let $\mathcal{E}_\upsilon^\sigma$ be the set of arcs associated with virtual leader, and denote $\bar{\mathcal{E}}^\sigma=\mathcal{E}^\sigma\cup\mathcal{E}_\upsilon^\sigma$, where $\cup$ is the union operation.

To further formulate the network considering in this work, we give the following assumption.
\begin{assumption}\label{assum-n}
	Consider a communication network including all followers and a leader (virtual reference system). We assume that:\\ {\color{blue}
	1)~The communication network is real-time disconnected and jointly connected, i.e., the union of all possible topologies $\bar{\mathcal{E}}_\cup=\bigcup_{\sigma(t)\in\mathcal{P}}\bar{\mathcal{E}}^\sigma$ contains a spanning tree rooted at the leader but for any $\sigma(t)\in\mathcal{P}$, $\bar{\mathcal{E}}^\sigma$ does not contain a spanning tree; \\
	2)~For every communication link in $\bar{\mathcal{E}}_\cup$, the time from when it becomes available (established) to when it is interrupted by a failure or an attack follows a negative exponential distribution with parameter $\lambda$;\\
	3)~	For every communication link in $\bar{\mathcal{E}}_\cup$, the length of time that the link remains interrupted (due to a failure or attack) follows a negative exponential distribution with parameter $\mu$.}
\end{assumption}


Note that in the above scenario, a spanning tree may never exist at any time. Taking this worst‑case perspective, the main problem of this paper is: how to design a pinning controller to achieve synchronization under switching networks that lack a spanning tree at every instant, especially when the leader system is open‑loop unstable?

Before solving this main problem, it is necessary to give some definitions, lemmas, and assumptions.
\begin{definition}
	Let $A\in\mathbb{R}^{n\times n}$ be a matrix whose off-diagonal elements are non-positive. Then, matrix $A$ is called a M-matrix if all its principal minors are non-negative. Moreover, we call $A$ a non-singular M-matrix if all its principal minors are positive.
\end{definition}

\begin{lemma}[\cite{Cooperative2009Qu}]\label{M1}
	The matrix $A\in\mathbb{R}^{n\times n}$ is a non-singular M-matrix if and only if there is a positive diagonal matrix $\Theta$ such that $\Theta A+A^T\Theta>\varepsilon I_n$ where $\varepsilon$ is a positive constant.
\end{lemma}

\begin{lemma}[\cite{Cooperative2009Qu}]\label{M2}
	Let $A\in\mathbb{R}^{n\times n}$ be a M-matrix. The matrix inequality $\Theta A+A^T\Theta\geq 0$ holds with a positive semi-definite solution $\Theta\geq 0$ if there exists a positive vector $x$ such that $Ax\geq 0$.  
\end{lemma}
\begin{assumption}\label{assum1}
	Assume each element of the the vector function $f(x)$ is
	Lipschitz with constant $\rho$.
\end{assumption}

\begin{assumption}\label{assum2}
	The pair $(A,B)$ is controllable. 
\end{assumption}

\begin{remark}
	Lemma \ref{M1} and Lemma \ref{M2} are rephrased by several theorems in \cite[Chapter 4]{Cooperative2009Qu}. Readers could refer to this monograph for their proof. Assumption \ref{assum-n} describes the communication mode among nodes. Assumption \ref{assum1} and Assumption \ref{assum2} are mild because they are conventional assumptions in nonlinear multi-agent systems. 
\end{remark}


\begin{remark}\label{probability}
	\color{blue}Under Assumption~\ref{assum-n}, for any node $w\in\mathcal{V}$, the union graph $\bar{\mathcal{E}}_\cup$ (the fully healthy network) contains a directed path $P_w$ from the reference system $v$ to $w$. Let the arcs on $P_w$ be $p_1,\dots,p_{m_w}$. Each arc $p_i$ independently alternates between \textit{up} and \textit{down} states: the up duration follows a negative exponential distribution with parameter $\lambda$, and the down duration follows a negative exponential distribution with parameter $\mu$, with $\lambda,\mu>0$. For a single arc, the probability it stays up is $\frac{\mu}{\lambda+\mu}$. Because the arcs are independent, the probability of all arcs on $P_w$ are simultaneously up is $(\frac{\mu}{\lambda+\mu})^{m_w} > 0$. Hence, the path $P_w$ is fully operational for a strictly positive fraction of time. This conclusion is crucial for the proof of the final conclusions in Theorem \ref{thm3} and Theorem \ref{thm1} of this paper.

\end{remark}

\section{Synchronization under switching topologies with fixed coupling gain}\label{sec3}

This section will first propose a network transformation mapping as well as its function and properties. Then, we will show how network transformation mapping help achieve synchronization under the unbounded reference system and the real-time disconnected switching networks.

\subsection{Network transformation mapping}\label{sec3.1}

This subsection intends to give the design of control protocol with network transformation mapping in the situation where the communication network does not have a directed spanning tree at any time. This protocol is proven to treat the traditional method as a special case.

Designing the pinning control law as
\begin{align}
&u_i=\kappa_i\mathcal{K}\left(\sum_{j=1}^N\beta_{ij}^\sigma\left(x_j-x_i\right)+\iota_i^\sigma\left(x_0-x_i\right)\right),\label{u1}
\end{align}
where $\kappa_i$ is the coupling gain with respect to agent $i$; $\mathcal{K}$ is the feedback gain that will be designed later. Moreover, {\color{blue}$\beta_{ij}^\sigma$ is a modified element of the adjacency matrix as a function of $\sigma (t)$, which is calculated from the original agency matrix parameters $\alpha_{ij}^\sigma$ and the network transformation mapping parameters $\delta_{ij}^{\sigma}$, i.e.,
\begin{align}
&\beta_{ij}^\sigma=\alpha_{ij}^\sigma\delta_{ij}^\sigma,\label{b1}\\
&\delta_{ij}^\sigma=\ell_{i}^\sigma\wedge\ell_{j}^\sigma,\label{b2}
\end{align} 
where $\ell_{i}^\sigma=\bigvee_{k_j\in\mathcal{I}^\sigma}\ell_{ik_j}^\sigma$ and $\ell_{ik_j}^\sigma=1$ means there is a path pointing from $k_j$ to $i$ ($k_j$ belongs to an index set of pinning nodes $\mathcal{I}^\sigma\triangleq\{k_1,\ldots,k_{m^\sigma}\}$) and $\ell_{ik_j}^\sigma=0$ otherwise. }

{\color{blue}Now, we introduce how to compute the network transformation mapping parameter $\delta_{ij}^\sigma$. According to the knowledge of graph theory, there is a path pointing from $k_j$ to $i$ with $2$ hops if and only if there is $l\in\mathcal{V}$ such that $\alpha_{il}^\sigma=1$ and $\alpha_{lk_j}^\sigma=1$. Hence,  $\alpha_{ik_j}^\sigma(2)\triangleq\sum_{l=1}^N\alpha_{il}^\sigma\alpha_{lk_j}^\sigma>0$ means there is a least one path pointing from $k_j$ to $i$ with $2$ hops. By the similar analysis, one may obtain that $\alpha_{ik_j}^\sigma(h)$---the element of $(\mathcal{A}^\sigma)^h$---represents the existence of path from $k$ to $i$ with $h$ hops. Now, we denote $\bar{\ell}_{ik_j}^\sigma$ the element of $\sum_{l=1}^{N}(\mathcal{A}^\sigma)^{l-1}$. Then, $\bar{\ell}_{ik_j}^\sigma>0$ means there is a path pointing from $k_j$ to $i$ with at least $1$ hop and at most $N-1$ hops. For convenience, we construct $\ell_{i}^\sigma=\bar{\ell}_{ik}^\sigma\wedge 1$, and hence $\delta_{ij}^\sigma$ can be calculated by (\ref{b2}).} Herein, we provide an example to help readers understand how to calculate the network transformation mapping.

{\color{cyan}Consider the original communication network shown in Fig. \ref{calexample}a, which consists of a leader (virtual reference system, labeled by $0$) and four follower nodes. First, each node computes its reachability indicator \(\ell_{i}\): \(\ell_{i}=1\) if there exists a directed path from pinning node to node \(i\); otherwise \(\ell_i=0\). In Fig. \ref{calexample}a, node 1 is a pinning node, so \(\ell_1=1\); node 2 can be reached via \(1\to 2\), so \(\ell_2=1\); nodes 3 and 4 have no path from any pinning node, so \(\ell_{3}=0,\ \ell_{4}=0\). Then, $\delta_{21}=\ell_{2}\wedge\ell_{1}=1\wedge 1=1$, so $\beta_{21}=\alpha_{21}\delta_{21}=1$. Similarly, we can calculate $\beta_{23}=\alpha_{23}\delta_{23}=1\wedge 0=0$ and $\beta_{43}=\alpha_{43}\delta_{43}=1\wedge 0=0$. Therefore, the transformed network is shown in Fig. \ref{calexample}b.}

\begin{figure}[!t]
	\centering
	\includegraphics[width=4cm]{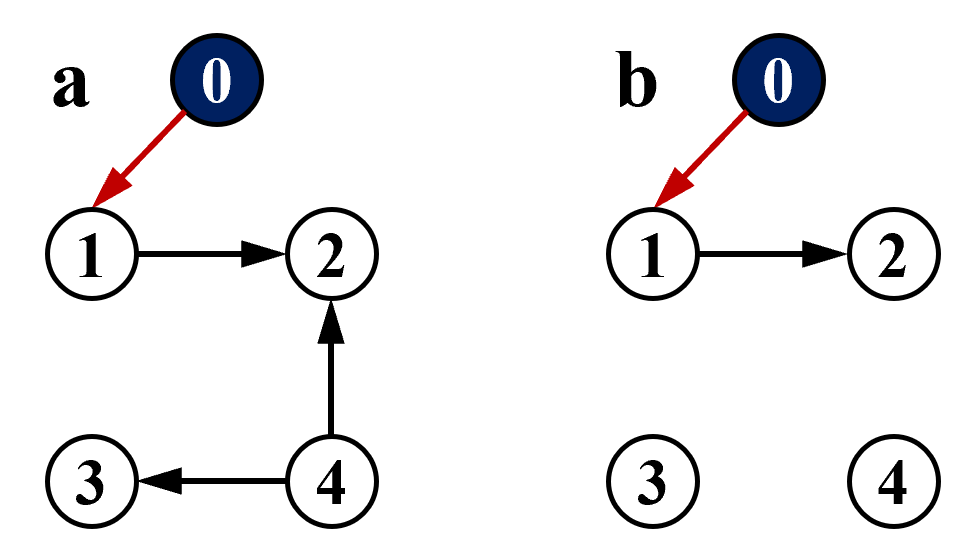}\\
	\caption{Switching networks for counterexample}\label{calexample}
\end{figure}

The rationale behind network transformation mapping (\ref{b1})--(\ref{b2}) is as follows. When agent $i$ becomes a descendant of the reference system at a specific moment, its error dynamics relative to the reference system are theoretically expected to exhibit a convergent trend. However, if agent $i$ simultaneously receives information of agent $j$ (which is not a descendant of the reference system so has a divergent trend), then the tracking error of agent $j$ will be transmitted to agent $i$, resulting in the tracking error of agent $i$ also becoming divergent. This means that even though agent $i$ indirectly accesses reference information, its tracking error still fails to have convergence trend. To address this issue, the proposed network transformation mapping (\ref{b1})--(\ref{b2}) is employed to help each agent dynamically evaluates the topological relationships among agents and selectively filters neighbor information. Specifically, it ensures that each agent only incorporates information from neighbors that belong to the reference system's descendant subgraph. This avoids the agent from being unable to track the reference system due to receiving incorrect information.
\begin{figure}[!t]
	\centering
	\includegraphics[width=8cm]{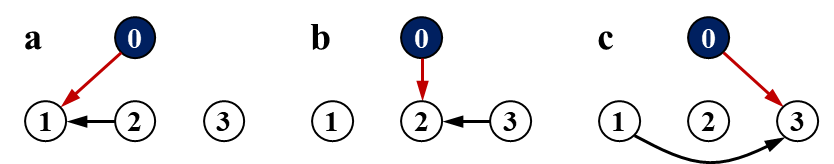}\\
	\caption{Switching networks for counterexample}\label{counterexample}
\end{figure}

{\color{red}Now, we will illustrate with a simple example that without a network transformation mapping, there may be a situation where all nodes are unable to track the leader forever (even if they can all obtain information about the leader at some point in time). Consider a switching sequence that cycles among three topologies shown in Fig. \ref{counterexample}, with equal dwell time. When the network is Fig. \ref{counterexample}a, Agent 1 can receive the leader's information but is simultaneously disturbed by erroneous information from Agent 2; thus, it cannot track the leader during this interval. When the network switches to Fig. \ref{counterexample}b, Agent 2 can receive the leader's information but is disturbed by Agent 3; hence, Agent 2 cannot track the leader during that interval. Moreover, because the leader's state is unbounded, Agent 1 tends to diverge further during this interval. When the network switches to Fig. \ref{counterexample}c, the erroneous information from Agent 1 prevents Agent 3 (which otherwise has access to the leader) from tracking the leader. As a result, although this switching sequence is a special case of the negative exponential distribution and satisfies Assumption 1, the system fails to achieve synchronization due to the absence of the network transformation mapping to prevent contamination by erroneous information.}

Now, we introduce the functions of (\ref{b1}): it can decompose the communication network into two separated subgraphs (Lemma \ref{trans}).

\begin{lemma}\label{trans}
	The network corresponding to adjacency matrix $\mathcal{B}^\sigma=[\beta_{ij}^\sigma]_{i,j=1}^N$ {\color{blue}can be} divided into two independent categories $\mathcal{G}_\star^\sigma$ and $\mathcal{G}_\diamond^\sigma$, where subgraph $\mathcal{G}_\star^\sigma$ contains all the pinning nodes belonging to $\mathcal{I}^\sigma$ and their descendants, and $\mathcal{G}_\diamond$ includes the other nodes. Moreover, the transformation leaves the arcs between the pair belonging to $\mathcal{G}_\star^\sigma$ as {\color{blue}they are}, and removes the arcs belonging to $\mathcal{G}_\diamond^\sigma$ as well as the arcs among $\mathcal{G}_\star^\sigma$ and $\mathcal{G}_\diamond^\sigma$. 
\end{lemma}

\begin{IEEEproof}
	See in Appendix \ref{app-L3}.
\end{IEEEproof}
{\color{red}Moreover, from the definition of the network transformation mapping, the following properties hold. (i) Uniqueness: For a given original graph $\mathcal{G}^\sigma$ and pinning node set $\mathcal{I}^\sigma$, the transformed adjacency matrix $\mathcal{B}^\sigma=[\beta_{ij}^\sigma]_{i,j=1}^N$ is uniquely determined, because $\ell_{ik}^\sigma$ is uniquely given by reachability and $\delta_{ij}^\sigma = \ell_i^\sigma \wedge \ell_j^\sigma$ is a deterministic logical operation. (ii) Invariance under switching: The network transformation mapping is applied independently at each switching instant, and the decomposition is always well-defined. (iii) Relation to maximal reachable subgraph: The subgraph $\mathcal{G}_\star^\sigma$ induced by the nodes with $\ell_{ik}^\sigma = 1$ is exactly the maximal reachable subgraph of the original graph with respect to the pinning node set, and it retains all original arcs among these nodes; the complementary subgraph $\mathcal{G}_\diamond^\sigma$ contains the remaining nodes and has no arcs. Hence, the network transformation mapping decomposes the original graph into the maximal reachable subgraph and a set of isolated nodes.}

Now, some notations based on the transformed network should be given. The details of two subgraphs are noted as $\mathcal{G}_\star^\sigma=\{\mathcal{V}_\star^\sigma,\mathcal{E}_\star^\sigma\}$ and $\mathcal{G}_\diamond^\sigma=\{\mathcal{V}_\diamond^\sigma, \mathcal{E}_\diamond^\sigma\}$ with $\mathcal{V}_\star^\sigma,  \mathcal{V}_\diamond^\sigma, \mathcal{E}_\star^\sigma, \mathcal{E}_\diamond^\sigma$ being nodes set and arcs set respectively. Let $\mathcal{B}_{\star}^\sigma$ and $\mathcal{B}_\diamond^\sigma$ with respect to $\beta_{ij}^\sigma$ be their adjacency matrices, while the corresponding Laplacian matrices are defined as $\mathcal{L}_\star^\sigma$ and $\mathcal{L}_\diamond^\sigma$. Since all the pinning nodes belong to subgraph $\mathcal{G}_\star^\sigma$, the pinning matrix can be designed as $\mathcal{S}_\star^\sigma=diag\{\iota_i^\sigma,~i\in\mathcal{V}_\star^\sigma\}$.

At the end of this subsection, it is important to point out that the calculation of the transformed network requires the agent to know the adjacency matrix of the original network. It is a mild requirement because the eigenvalue of the Laplacian matrix is often used when calculating the coupling gain of multi-agent systems. And the calculation of its eigenvalue also needs to know the information of the whole adjacency matrix. In addition, Algorithm \ref{alg1} gives a completely distributed method to calculate $\delta_{ij}^\sigma$, in which each node only needs the information about itself and its neighbor nodes to obtain $\delta_{ij}^\sigma$. {\color{blue}In this algorithm, $\alpha_{ij}^\sigma[\iota]$ represents the reachability from node $j$ to node $i$ with $\iota$ hops.} For example, $\alpha_{ij}^\sigma[\iota]=1$ means there is a path pointing from $j$ to $i$ with $\iota$ hops. 

\begin{algorithm}
	\caption{Completely distributed calculation method of $\delta_{ij}^\sigma$} 
	\label{alg1}
	\begin{algorithmic}[1]
		\REQUIRE Node $i$, and the information of $\alpha_{ij}^\sigma$. 
		\ENSURE $\delta_{ij}^\sigma$.
		\STATE Set $\iota=1$, and {\color{blue}$\alpha_{ij}^\sigma[1]=\alpha_{ij}^\sigma$}.
		\STATE Find a set $\mathcal{N}_i[1]=\{j~|~\alpha_{ij}^\sigma[1]=1\}$;
		\STATE Calculate 
		$\ell_{i{k}}^\sigma=\bigvee_{j\in\mathcal{I}}\alpha_{ij}^\sigma[1]$;
		\WHILE {$\ell_{i{k}}^\sigma \neq 1$}
		\STATE Let $\iota=\iota+1$;
		\STATE Calculate $$\color{blue}\alpha_{ij}^\sigma[\iota]=\bigvee_{l\in\{i\}\cup\mathcal{N}_{i}[1]}\alpha_{il}^\sigma[1]\alpha_{lj}^\sigma[\iota-1];$$
		\STATE Calculate $\ell_{i{k}}^\sigma=\bigvee_{j\in\mathcal{I}}\alpha_{ij}^\sigma[\iota]$;
		\ENDWHILE
		\STATE Calculate $\delta_{ij}^\sigma$ by (\ref{b2}).
	\end{algorithmic}
\end{algorithm}



\subsection{Stability of the synchronization protocol}\label{sec3.2}

This subsection will first focus on the convergence rate of $\mathcal{V}_\star^\sigma$ concerning coupling gain. Then, based on the estimation of the measure of convergence time interval, we will design control parameters such that the convergence amplitude of the agent during the convergence interval is greater than the divergence amplitude during the divergence interval. This means the sufficient conditions for guaranteeing the synchronization under switching topologies will be proved. Now, one may state the following.

\begin{theorem}\label{thm2}
	Given the multi-agent systems (\ref{sys}) with control law (\ref{u1}) subjected to Assumption \ref{assum1} and \ref{assum2}. {\color{blue}Consider any time interval over which the switching signal \(\sigma(t)\) is a constant, i.e., no topology change occurs during this interval.} Then, the convergence rate of synchronization of all nodes in $\mathcal{V}_\star^\sigma$ is at least
	\begin{align}\label{rela}
	\wp_\star=&\frac{2}{\bar{\kappa}_\star^\sigma\bar{\theta}_\star^\sigma}\left(-\bar{\kappa}_\star^\sigma\underline{\kappa}_\star^\sigma\underline{\theta}_\star^\sigma+\bar{\kappa}_\star^\sigma\left(\bar{\theta}_\star^\sigma\right)^2\right)\notag\\
	&+4\sqrt{N}\rho\bar{\lambda}(P)\bar{\lambda}(\mathcal{H}_\star^\sigma)\underline{\lambda}(\mathcal{H}_\star^\sigma),
	\end{align}
	where $\bar{\kappa}_\star^\sigma=\max_i\{\kappa_i,~i\in\mathcal{V}_\star^\sigma\}$, $\underline{\kappa}_\star^\sigma=\min\{\kappa_i,~i\in\mathcal{V}_\star^\sigma\}$; $\bar{\theta}_\star^\sigma=\max\{\theta_i^\sigma,~i\in\mathcal{V}_\star^\sigma\}$, $\underline{\theta}_\star^\sigma=\min\{\theta_i^\sigma,~i\in\mathcal{V}_\star^\sigma\}$ with $\theta_i^\sigma$ being the $i$th diagonal element of $\Theta_\star^\sigma$; $\Theta_\star^\sigma$ is chosen such that $sym\{\Theta_\star^\sigma\mathcal{H}_\star^\sigma\}$ is positive definite diagonal matrix; $\kappa_*$ is a positive constant and $\mathcal{H}_\star^\sigma=\mathcal{L}_\star^\sigma+\mathcal{S}_\star^\sigma$; $P$ is a positive definite matrix solved by
		\begin{align}
			P(A-\kappa_*BK)+(A-\kappa_*BK)^TP=-\kappa_*I_n,
		\end{align}
		and $K$ is chosen such that $A-\kappa_*BK$ is a Hurwitz matrix.
\end{theorem}

\begin{IEEEproof}
See in Appendix \ref{app-Th1}.
\end{IEEEproof}

This theorem pointed out the convergence rate of synchronization for the agents in $\mathcal{V}_\star^\sigma$. However, the agents belonging to $\mathcal{V}_\diamond^\sigma$ have no access to the information of pinning nodes. Hence, the performance of their error dynamics depends on the eigenvalues of system matrix $A$. For this reason, many studies required multi-agent systems to be open-loop marginally stable \cite{zhang2023DSESWTICH,Lu2020Robust,10323478,liu2024distributed,ZHANG2024111755}. However, this paper considers the case where agents in $\mathcal{V}_\diamond^\sigma$ evolve with unstable dynamics. Consequently, the $i$th agent's dynamics can converge to $x_0$ in the time interval where $i\in\mathcal{V}_\star^\sigma$ and maybe away from $x_0$ otherwise. 
To overcome this issue, we need to guarantee that the convergence range of $e_i$ in the time interval of $i\in\mathcal{V}_\star^\sigma$ should be greater than the divergence range during others. To this end, the estimation of the measure of convergence period is necessary. According to Remark \ref{probability}, the probability that all arcs on the path remain up is $(\mu/(\lambda+\mu))^N\triangleq m_c>0$ because the length of a path from virtual leader to an agent is at most $N$. Now, we can state the following theorem.

\begin{theorem}\label{thm3}
	Given the switching topologies subject to Assumption \ref{assum-n}. Then, multi-agent systems (\ref{sys}) satisfying Assumption \ref{assum1} and \ref{assum2} can achieve synchronization against switching topologies without including spanning tree at any time if the control protocol (\ref{u1}) with network transformation mapping (\ref{b1}) is implemented and \textcolor{blue}{the convergence rate \(\wp_\star\) defined in Theorem~1 is designed such that}
	\begin{align}
		\wp_\star>\frac{1-m_c}{m_c}\wp_\diamond,
	\end{align}
	where
	\begin{align}
		\wp_\diamond=\frac{\bar{\kappa}_\diamond^\sigma\bar{\theta}_\diamond^\sigma}{\underline{\kappa}_\diamond^\sigma\underline{\theta}_\diamond^\sigma}\left(\frac{1}{2}\bar{\lambda}(sym\{A\})+\rho\right),
	\end{align} 
	with $\bar{\kappa}_\diamond^\sigma=\max\{\kappa_i,~i\in\mathcal{V}_\diamond^\sigma\}$, $\bar{\theta}_\diamond^\sigma=\bar{\lambda}(\Theta_\diamond^\sigma)$, $\underline{\kappa}_\diamond^\sigma=\min\{\kappa_i^{-1},~i\in\mathcal{V}_\diamond^\sigma\}$, and $\underline{\theta}_\diamond^\sigma=\underline{\lambda}(\Theta_\diamond^\sigma)$.
\end{theorem}

\begin{IEEEproof}
This proof will be completed in two steps. 

1) \textcolor{blue}{First, we give an upper bound on the divergence rate of $e_i$ for $i\in\mathcal{V}_\diamond^\sigma$.}  Since all pinning nodes belong to $\mathcal{V}_\star^\sigma$ and there is no arcs between $\mathcal{V}_\diamond^\sigma$ and $\mathcal{V}_\star^\sigma$, we know $\iota_i^\sigma=0$ for all $i\in\mathcal{V}_\diamond^\sigma$. Let $e_\diamond=col\{e_i,~i\in\mathcal{V}_\diamond^\sigma\}$, then
\begin{align}
\dot{e}_\diamond&=\left(I_{\pi'(\sigma)}\otimes A\right)e_\diamond+\tilde{f}_\diamond(\bar{x}_\diamond,\bar{x}_{0,\diamond})-(\Upsilon_\diamond^\sigma\mathcal{L}_\diamond^\sigma\otimes B\mathcal{K})e_\diamond,
\end{align}
where $\tilde{f}_\diamond(\bar{x}_\diamond,\bar{x}_{0,\diamond})=col\{\tilde{f}_i(x_i,x_0),~i\in\mathcal{V}_\diamond^\sigma\}$ with $\tilde{f}_i(x_i,x_0)=f(x_i)-f(x_0)$; $\pi'(\sigma)$ represents the cardinal number of $\mathcal{V}_\diamond^\sigma$, and $\Upsilon_\diamond^\sigma=diag\{\kappa_i,~i\in\mathcal{V}_\diamond^\sigma\}$. 

Using the conclusion of Lemma \ref{M2}, we know $\mathcal{L}_\diamond^\sigma$ is a singular M-matrix and satisfies $\mathcal{L}_\diamond^\sigma 1_{\pi'(\sigma)}\geq 0$. Hence, there exists a positive semi-definite matrix $\Theta_\diamond^\sigma$ such that
\begin{align}
\Theta_\diamond^\sigma\mathcal{L}_\diamond^\sigma+\left(\mathcal{L}_\diamond^\sigma\right)^T\Theta_\diamond^\sigma\geq 0.
\end{align}
Then, by constructing the Lyapunov candidate function as $V_2=e_\diamond^T\left(\Theta_\diamond^\sigma\Upsilon_\diamond^\sigma\otimes I_{\pi'(\sigma)}\right)e_\diamond$, one has
\begin{align}
\dot{V}_2=&e_\diamond^T\left(\Theta_\diamond^\sigma\Upsilon_\diamond^\sigma\otimes (A+A^T)\right)e_\diamond\notag\\
&+2e_\diamond^T\left(\Theta_\diamond^\sigma\Upsilon_\diamond^\sigma\otimes I_{\pi'(\sigma)}\right)\tilde{f}_\diamond(\bar{x}_\diamond,\bar{x}_{0,\diamond})\notag\\
&-\left(\Upsilon_\diamond^\sigma\left(\Theta_\diamond^\sigma\mathcal{L}_\diamond^\sigma+\left(\mathcal{L}_\diamond^\sigma\right)^T\Theta_\diamond^\sigma\right)\Upsilon_\diamond^\sigma\otimes BKK^TB^T\right)e_\diamond\notag\\
\leq&\left(\bar{\kappa}_\diamond^\sigma\bar{\theta}_\diamond^\sigma\bar{\lambda}(sym\{A\})+2\rho\sqrt{N}\bar{\kappa}_\diamond\bar{\theta}^\sigma\right)\|e_\diamond\|^2\notag\\
\leq& 2\wp_\diamond V_2(t).
\end{align}
Thus,
\begin{align}
V_2(t)\leq V_2(t_m)\exp\{2\wp_\diamond(t-t_m)\},
\end{align}
where notation $t_m$ has the same meaning as that in Theorem \ref{thm2}. Within the definition of $V_2(t)$, we further obtain
\begin{align}
&\frac{1}{2}\underline{\kappa}_\diamond^\sigma\underline{\theta}_\diamond^\sigma\|e_\diamond(t)\|^2\leq V_2(t)\leq V_2(t_m)\exp\{2\wp_\diamond(t-t_m)\}\notag\\
&\leq \frac{1}{2}\bar{\kappa}_\diamond^\sigma\bar{\theta}_\diamond^\sigma\|e_\diamond(t_m)\|^2\exp\{2\wp_\diamond(t-t_m)\},
\end{align} 
which leads to
\begin{align}
\|e_i(t)\|\leq\sqrt{\frac{\bar{\kappa}_\diamond^\sigma\bar{\theta}_\diamond^\sigma}{\underline{\kappa}_\diamond^\sigma\underline{\theta}_\diamond^\sigma}}\|e_i(t_m)\|\exp\{\wp_\diamond(t-t_m)\},
\end{align}
for $i\in\mathcal{V}_\diamond^\sigma$.

2) Second, we prove sufficient conditions for the synchronization against switching topologies without a spanning tree at any time.

Let $t_\tau$ be the time when a failure or attack occurs or a communication link is repaired, where $\tau=1,2,\ldots$ Denote $\beth_{i,\tau}=\wp_\star$ and $\vartheta_{i,\tau}=\sqrt{\bar{\kappa}_\star^\sigma\bar{\theta}_\star^\sigma\bar{\lambda}(\mathcal{H}_\star^\sigma)\bar{\lambda}(P)/\underline{\kappa}_\star^\sigma\underline{\theta}_\star^\sigma\underline{\lambda}(\mathcal{H}_\star^\sigma)\underline{\lambda}(P)}$ if $i\in\mathcal{V}_{\star}^\sigma$ with $t\in[t_\tau,t_{\tau+1})$ and also denote $\beth_{i,\tau}=\wp_\diamond$ and $\vartheta_{i,\tau}=\sqrt{\bar{\kappa}_\diamond^\sigma\bar{\theta}_\diamond^\sigma/\underline{\kappa}_\diamond^\sigma\underline{\theta}_\diamond^\sigma}$ otherwise.

For a large enough $M>0$ and $t\in(t_{M},t_{M+1})$, we have
\begin{align}\label{dt}
\|e_i(t)\|\leq& \vartheta_{i,M}\|e_i(t_M)\|\exp\{\beth_{i,M}(t-t_M)\}\notag\\
\leq&\vartheta_{i,M}\vartheta_{i,M-1}\|e_i(t_{M-1})\|\notag\\
&\times\exp\{\beth_{i,M}(t-t_M)\}\exp\{\beth_{i,M-1}(t_M-t_{M-1})\}\notag\\
\leq& \|e_i(t_0)\|\exp\{\beth_{i,M}(t-t_M)\}\notag\\
&\times\prod_{j=0}^M\vartheta_{i,j}\times \prod_{j=0}^M\exp\{\beth_{i,j}(t_{j+1}-t_j)\}\notag\\
\leq&\vartheta_i\|e_i(t_0)\|\exp\{-m_c\wp_\star t\}\exp\{(1-m_c)\wp_\diamond t\}\notag\\
=&\vartheta_i\|e_i(t_0)\|\exp\left\{\left(-m_c\wp_\star+(1-m_c)\wp_\diamond\right)t\right\},
\end{align}
where $\vartheta_i=\prod_{j=0}^M\vartheta_{i,j}$. Since $\wp_\star$ is a parameter that can be designed arbitrarily, one may choose
\begin{align}
	\wp_\star>\frac{1-m_c}{m_c}\wp_\diamond
\end{align}
such that the error dynamic converge to zero with fixed coupling gain under the switching topologies without a spanning tree at any time.
\end{IEEEproof}

\begin{remark}
	\color{red}Note that none of the proofs in this paper rely on the memoryless property of the exponential distribution. The only property used is the positive probability that a link stays up, i.e., $\frac{\mu}{\lambda+\mu}$ (as stated in Remark~\ref{probability}). Actually, the negative exponential distribution in Assumption~1 is adopted because it is the most common model for the failure rate of electronic components and communication links. Consequently, Assumption~\ref{assum-n} can be extended to any independent, ergodic renewal or Markov switching processes, as long as each link has a well-defined positive probability of being up. In such extensions, the constant $m_c = (\mu/(\lambda+\mu))^N$ in Theorem~\ref{thm3} should be replaced by the product of the individual links' up probabilities, and all conclusions remain valid.
\end{remark}

\begin{remark}\label{rem-increase}
	In light of the proof of Theorem \ref{thm3}, we know that the realization of synchronization mainly depends on the convergence of the agents belonging to $\mathcal{V}_\star^\sigma$. It is shown in Theorem \ref{thm2} that the agents in $\mathcal{V}_\star^\sigma$ achieve synchronization with convergence rate $\wp_\star$. And $\wp_\star$ increases with $\bar{\kappa}_\star$ and $\underline{\kappa}_\star$ increasing. Hence, a group of larger coupling gains may improve the performance of synchronization. {\color{blue}Moreover, even if $m_c$ is very small (for example, under severe attack scenarios), we can still ensure synchronization by increasing coupling gains.}
\end{remark}

\begin{remark}
	It is noticed that both $m_c$ and $\wp_\star$ are difficult to be calculated. The statements in Remark \ref{rem-increase} infers the probability of using adaptive coupling gain. Therefore, Section \ref{sec4} will focus on this problem, and the results in this Section will play an important role.  
\end{remark}

\section{Synchronization under switching topologies with adaptive coupling gain}\label{sec4}

This section focuses on the synchronization of multi-agent systems under switching topologies with adaptive coupling gain. In the beginning, we give the corresponding adaptive control law:
\begin{align}
	&u_i=\left(\gamma_i+\omega_i\right)\mathcal{K}\left(\sum_{j=1}^N\beta_{ij}^\sigma\left(x_j-x_i\right)+\iota_i\left(x_0-x_i\right)\right),\label{u2}\\
	&\dot{\gamma}_i=\omega_i=\left\|\sum_{j=1}^N\beta_{ij}^\sigma\left(x_j-x_i\right)+\iota_i\left(x_0-x_i\right)\right\|_P^2,\label{u3}
\end{align}
where $\gamma_i+\omega_i$ are coupling gains that play the same role as $\kappa_i$ in (\ref{u1}); $\mathcal{K}\in\mathbb{R}^n$ is the feedback gain that will be designed later.

\begin{theorem}\label{thm1}
	Consider the multi-agent systems (\ref{sys}) subjected to Assumption \ref{assum1}--\ref{assum2}. We design the control law (\ref{u2}) with adaptive coupling gain (\ref{u3}) by choosing feedback gain $\mathcal{K}=KK^TB^TP$, and a symmetric positive definite $P$ satisfying
	\begin{align}
		sym\{P(A-\gamma_*BK)\}=-\gamma_*I_n,
	\end{align}
	where {\color{blue}$K$ is chosen such that $A-\gamma_*BK$ is Hurwitz stable and $\gamma_*$ is an arbitrary constant satisfying $\gamma_*>1$}. Then, synchronization of multi-agent system (\ref{sys}) can be achieved against the switching topologies described in Assumption \ref{assum-n}.
\end{theorem}

\begin{IEEEproof}
	The proof of this Theorem can be divided into two portions. First, we will show the $i$th agent's states $x_i$ converge to $x_0$ asymptotically with adaptive coupling gain for $i\in\mathcal{V}_{\star}^\sigma$. Second, it is shown that the adaptive coupling gain will gradually improve the convergence rate of nodes in $\mathcal{V}_{\star}^\sigma$, and then drive the synchronization of the whole system.
	
    Part 1):~Synchronization in $\mathcal{V}_{\star}^\sigma$ with adaptive gains. We also denote $e_i$ the synchronization error of the $i$th agent. Within the same notations defined in Theorem \ref{thm2}, we have $\xi_\star=(\mathcal{H}_\star^\sigma\otimes I_n)e_\star$. Hence, in the time interval where $\sigma(t)$ is a constant, we obtain
	\begin{align}\label{xi1}
		\dot{\xi}_\star=&(\mathcal{H}_\star^\sigma\otimes I_n)\dot{e}_\star\notag\\
		=&(I_{\pi(\sigma)}\otimes A)\xi_\star-\left(\mathcal{H}_\star^\sigma\Pi_\star^\sigma\otimes B\mathcal{K}\right)\xi_\star\notag\\
		&+(\mathcal{H}_\star^\sigma\otimes I_n)\tilde{f}_\star(\bar{x}_\star,\bar{x}_{0,\star}),
	\end{align}
	where $\Pi_\star^\sigma$ is given by $diag\{\pi_i,~i\in\mathcal{V}_\star^\sigma\}$ with $\pi_i=\gamma_i+\omega_i$. 
	
	To move on, denote a constant $\mathcal{M}=\bar{\kappa}\triangleq\max_i\{\kappa_i,~i=1,\ldots,N\}$. First, consider the scenario where $\gamma_i>\mathcal{M}$. Then, we have $\pi_i\geq\bar{\kappa}+\|\xi_i\|_P>\bar{\kappa}$. By the results of Theorem \ref{thm2} and Theorem \ref{thm3}, a group of coupling gains that are greater than $\bar{\kappa}$ will lead to the synchronization of multi-agent systems. Moreover, the synchronization yields $\omega_i\to 0$ and thus $\dot{\gamma}_i$ tends to be a constant. 
	
	Then, suppose there is a constant $T$ such that $\gamma_i<\mathcal{M}$ for $t<T$. In this scene, we choose the Lyapunov candidate as
	\begin{align} V_\star(t)=\frac{1}{2}\sum_{i\in\mathcal{V}_\star^\sigma}\theta_i^\sigma\left(2\gamma_i+\omega_i\right)\xi_i^TP\xi_i+\theta_i^\sigma(\gamma_i-\theta_i^\sigma\gamma_*)^2, \notag
	\end{align}	
	where $\gamma_*>0$ is a given constant; $\theta_i^\sigma$ is the entry of a diagonal positive definite matrix $\Theta_\star^\sigma=diag\{\theta_i^\sigma,~i\in\mathcal{V}_\star^\sigma\}$, and which is designed so that $sym\{\Theta_\star^\sigma\mathcal{H}_\star^\sigma\}>\gamma_* I_{n\pi(\sigma)}$. Hence, we can derivate $V_\star(t)$ along with $\xi_\star$ and yield
	\begin{align}\label{dv}
		\dot{V}_\star(t)=&\xi_\star^T(\Pi_\star^\sigma\Theta_\star^\sigma\otimes P)\dot{\xi}_\star\notag\\
		&+\xi_\star^T(\Pi_\star^\sigma\Theta_\star^\sigma\otimes I_n)\xi_\star-\gamma_*\xi_\star^T(\Theta_\star^\sigma\Theta_\star^\sigma\otimes I_n)\xi_\star\notag\\
		=&\xi_\star^T(\Pi_\star^\sigma\Theta_\star^\sigma\otimes (PA+A^TP))\xi_\star\notag\\
		&-\xi_\star^T\left(\Pi_\star^\sigma\Theta_\star^\sigma\mathcal{H}_\star^\sigma\Pi_\star^\sigma\otimes PB\mathcal{K}\right)\xi_\star\notag\\
		&-\xi_\star^T\left(\Pi_\star^\sigma\left(\mathcal{H}_\star^\sigma\right)^T\Theta_\star^\sigma\Pi_\star^\sigma\otimes \mathcal{K}^TB^TP\right)\xi_\star\notag\\
		&+2\xi_\star^T(\Pi_\star^\sigma\Theta_\star^\sigma\mathcal{H}_\star^\sigma\otimes P)\tilde{f}_\star(\bar{x}_\star,\bar{x}_{0,\star})\notag\\
		&+\xi_\star^T(\Pi_\star^\sigma\Theta_\star^\sigma\otimes I_n)\xi_\star-\gamma_*\xi_\star^T(\Theta_\star^\sigma\Theta_\star^\sigma\otimes I_n)\xi_\star.
	\end{align}
	Since $\mathcal{K}=KK^TB^TP$, we know $PB\mathcal{K}=\mathcal{K}^TB^TP=PBKK^TB^TP$. Then, 
	\begin{align}
		\dot{V}_\star(t)=&\xi_\star^T(\Pi_\star^\sigma\Theta_\star^\sigma\otimes (PA+A^TP))\xi_\star\notag\\
		&-\xi_\star^T\left(\Pi_\star^\sigma sym\{\Theta_\star^\sigma\mathcal{H}_\star^\sigma\}\Pi_\star^\sigma\otimes PBKK^TB^TP\right)\xi_\star\notag\\
		&+2\xi_\star^T(\Pi_\star^\sigma\Theta_\star^\sigma\mathcal{H}_\star^\sigma\otimes P)\tilde{f}_\star(\bar{x}_\star,\bar{x}_{0,\star})\notag\\
		&+\xi_\star^T(\Pi_\star^\sigma\Theta_\star^\sigma\otimes I_n)\xi_\star-\gamma_*\xi_\star^T(\Theta_\star^\sigma\Theta_\star^\sigma\otimes I_n)\xi_\star\notag\\
		\leq&\xi_\star^T(\Pi_\star^\sigma\Theta_\star^\sigma\otimes (PA+A^TP))\xi_\star\notag\\
		&-\gamma_*\xi_\star^T\left(\left(\Pi_\star^\sigma\right)^2\otimes PBKK^TB^TP\right)\xi_\star\notag\\
		&+2\xi_\star^T(\Pi_\star^\sigma\Theta_\star^\sigma\mathcal{H}_\star^\sigma\otimes P)\tilde{f}_\star(\bar{x}_\star,\bar{x}_{0,\star})\notag\\
		&+\xi_\star^T(\Pi_\star^\sigma\Theta_\star^\sigma\otimes I_n)\xi_\star-\gamma_*\xi_\star^T(\Theta_\star^\sigma\Theta_\star^\sigma\otimes I_n)\xi_\star.
	\end{align}
	
	Then, by repeating the process of (\ref{t1dv1})--(\ref{t1dv2}) in the proof of Theorem \ref{thm2}, we have
	\begin{align}\label{dv1}
		&\xi_\star^T(\Pi_\star^\sigma\Theta_\star^\sigma\otimes (PA+A^TP))\xi_\star\notag\\
		-&\xi_\star^T\left(\Pi_\star^\sigma sym\{\Theta_\star^\sigma\mathcal{H}_\star^\sigma\}\Pi_\star^\sigma\otimes PBKK^TB^TP\right)\xi_\star\notag\\
		+&\xi_\star^T(\Pi_\star^\sigma\Theta_\star^\sigma\otimes I_n)\xi_\star-\gamma_*\xi_\star^T(\Theta_\star^\sigma\Theta_\star^\sigma\otimes I_n)\xi_\star\notag\\
		=&-(\gamma_*-1)\xi_\star^T(\Pi_\star^\sigma\Theta_\star^\sigma\otimes I_n)\xi_\star,
	\end{align}
	and
	\begin{align}\label{dv2}
		&\xi_\star^T(\Pi_\star^\sigma\Theta_\star^\sigma\mathcal{H}_\star^\sigma\otimes P)\tilde{f}_\star(\bar{x}_\star,\bar{x}_{0,\star})\notag\\
		\leq&\sqrt{N}\rho\mathcal{M}\bar{\theta}_\star^\sigma\bar{\lambda}(P)\bar{\lambda}(\mathcal{H}_\star^\sigma)\underline{\lambda}(\mathcal{H}_\star^\sigma)\|\xi_\star\|^2.
	\end{align}
	
	
	Substituting (\ref{dv1}) and (\ref{dv2}) into (\ref{dv}) yields
	\begin{align*}
		\dot{V}_\star\leq& -(\gamma_*-1)\xi_\star^T(\Pi_\star^\sigma\Theta_\star^\sigma\otimes I_n)\xi_\star\notag\\
		&+\sqrt{N}\rho\mathcal{M}\bar{\theta}_\star^\sigma\bar{\lambda}(P)\bar{\lambda}(\mathcal{H}_\star^\sigma)\underline{\lambda}(\mathcal{H}_\star^\sigma)\|\xi_\star\|^2\notag\\
		\leq&\left(-(\gamma_*-1)\underline{\gamma}\underline{\theta}_\star^\sigma+\wp_{2\star}\right)\|\xi_\star\|^2,
	\end{align*}
	where $\underline{\gamma}=\min_i\{\gamma_i,~i\in\mathcal{V}_\star^\sigma\}$ with $\gamma_i(0)$ being the initial value of $\gamma_i$; and $\wp_{2\star}=\sqrt{N}\rho\mathcal{M}\bar{\theta}_\star^\sigma\bar{\lambda}(P)\bar{\lambda}(\mathcal{H}_\star^\sigma)\underline{\lambda}(\mathcal{H}_\star^\sigma)$. Therefore, $\dot{V}_\star< 0$ if $\underline{\gamma}>\wp_\star/((\gamma_*-1)\underline{\theta}_\star^\sigma)$. 
	
	It indicates that $\xi_\star$ has an asymptotic convergence trend within the current time period if $\gamma_i>\min\{\mathcal{M},\wp_\star/((\gamma_*-1)\underline{\theta}_\star^\sigma)\}$ for all $i\in\mathcal{V}_\star^\sigma$. Furthermore, since the right side of the equation (\ref{u3}) is a radial unbounded function, $\gamma_i$ will continue increase until $\underline{\gamma}>\wp_\star/((\gamma_*-1)\underline{\theta}_\star^\sigma)$. Up to now, we have shown that the adaptive control law (\ref{u2})--(\ref{u3}) can make the tracking error of the agent in $\mathcal{V}_\star^\sigma$ have a convergence trend.
	
	Part 2):~Synchronization can be achieved by adaptive coupling gain and switching topologies. Since Part 1) has shown $x_i$ converges to $x_0$ when $i\in\mathcal{V}_{\star}^\sigma$, we can estimate the associated convergence rate by repeating the method in Theorem \ref{thm2} (Convergence infers the boundedness of $\pi_i$ when $i\in\mathcal{V}_{\star}^\sigma$ and thus one may estimate its convergence rate by regarding $\pi_i$ as a bounded constant at any time). Consequently, within (\ref{rela}) and Theorem \ref{thm3}, the multi-agent systems achieve synchronization against switching topologies without including any spanning tree if there exists a group of large enough $\gamma_i$ such that
\begin{align*}
	&\frac{2}{\bar{\gamma}_\star^\sigma\bar{\theta}_\star^\sigma}\left(\bar{\gamma}_\star^\sigma\underline{\gamma}_\star^\sigma\underline{\theta}_\star^\sigma-\bar{\gamma}_\star^\sigma(\bar{\theta}_\star^\sigma)^2\right)\notag\\
	+&4\sqrt{N}\rho\bar{\lambda}(P)\bar{\lambda}(\mathcal{H}_\star^\sigma)\underline{\lambda}(\mathcal{H}_\star^\sigma)>\frac{1-m_c}{m_c}\wp_\diamond\bar{\gamma}_\star^\sigma,
\end{align*}
where $\bar{\gamma}_\star^\sigma=\max\{\gamma_i,~i\in\mathcal{V}_\star^\sigma\}$ and $\underline{\gamma}_\star^\sigma=\min\{\gamma_i,~i\in\mathcal{V}_\star^\sigma\}$. This formula holds if
\begin{align}\label{sufficient}
	&\frac{2}{\bar{\theta}_\star^\sigma}\left(\bar{\gamma}\underline{\gamma}\underline{\theta}_\star^\sigma-\bar{\gamma}(\bar{\theta}_\star^\sigma)^2\right)
	+4\sqrt{N}\rho\bar{\lambda}(P)\bar{\lambda}(\mathcal{H}_\star^\sigma)\underline{\lambda}(\mathcal{H}_\star^\sigma)\notag\\
	&>\frac{1-m_c}{m_c}\frac{\bar{\theta}_\diamond^\sigma}{\underline{\theta}_\diamond^\sigma}\left(\frac{1}{2}\bar{\lambda}(sym\{PA\})+\rho\bar{\lambda}(P)\right)\cdot\frac{\bar{\gamma}^2}{\underline{\gamma}},
\end{align}
where $\bar{\gamma}=\max\{\gamma_i,~i=1,\ldots,N\}$, and $\underline{\gamma}=\min\{\gamma_i,~i=1,\ldots,N\}$. Note that the right half part of (\ref{u2}) is a radial unbounded function and indicates that the coupling gain $\gamma_i$ will increase monotonically down to the synchronization. Namely, the tracking error cannot converge to zero with a group of initial gain $\gamma_i$ that do not fulfill the sufficient condition (\ref{sufficient}). It will lead to the monotonical increase of $\gamma_i$ and the error dynamics begin to converge when (\ref{sufficient}) holds. Hence, the problem is transformed to show that equation (\ref{sufficient}) can be fulfilled with the increase of $\gamma_i$.

Since range of $\sigma(t)$ contains only a finite number of elements, we have $\bar{\theta}_\star^M\triangleq\max_t\{\bar{\theta}_\star^\sigma\}$, $\underline{\theta}_\star^m\triangleq\min_t\{\underline{\theta}_\star^\sigma\}$,  $\bar{\theta}_\diamond^M\triangleq\max_t\{\bar{\theta}_\diamond^\sigma\}$, and $\underline{\theta}_\diamond^m\triangleq\min_t\{\underline{\theta}_\diamond^\sigma\}$. Consequently, the following constants can be defined:
\begin{align}
	&\psi_1\triangleq \frac{2{\underline{\theta}_\star^m}}{\bar{\theta}_\star^M}, \\
	&\psi_2\triangleq4\sqrt{N}\rho\bar{\lambda}(P)\bar{\lambda}(\mathcal{H}_\star^\sigma)\underline{\lambda}(\mathcal{H}_\star^\sigma)-2\bar{\theta}_\star^M,\\
	&\psi_{3}\triangleq\frac{1-m_c}{m_c}\frac{\bar{\theta}_\diamond^M}{\underline{\theta}_\diamond^m}\left(\frac{1}{2}\bar{\lambda}(sym\{PA\})+\rho\bar{\lambda}(P)\right).
\end{align}
Then, a simplified form of (\ref{sufficient}) gives rise to
\begin{align}
	\psi_{1}\underline{\gamma}^2\bar{\gamma}+\psi_2\underline{\gamma}\bar{\gamma}-\psi_{3}\bar{\gamma}^2>0.
\end{align}
Because all $\gamma_i$s have the same dynamics, the properties of the above formula can be studied by an approximate function
\begin{align}
	\psi_1x^2+\psi_2x-\psi_3x>0.
\end{align}
It is not difficult to find that there is a constant $x_*$ so that the function is monotonic increasing for $x>x_*$. It indicates that (\ref{sufficient}) can be achieved with the increase of $\gamma_i$. {\color{blue}As a result, similar to the process in (\ref{dt}), we can obtain that $\|e_i(t)\|$ exhibits an exponential convergence trend once (\ref{sufficient}) holds.}

{\color{red}We further note that the adaptive gains $\gamma_i(t)$ will finally converge to a constant. From~(\ref{u3}), $\dot{\gamma}_i = \|\xi_i\|_P^2 \ge 0$ with $\xi_{i}=\sum_{j=1}^N\beta_{ij}^\sigma\left(x_j-x_i\right)+\iota_i^\sigma\left(x_0-x_i\right)$, so each $\gamma_i(t)$ is nondecreasing. In the previous proof, we have shown that once the coupling gains become sufficiently large to satisfy condition~(\ref{sufficient}), the tracking errors $e_i$ converge to zero exponentially. Consequently, $\xi_i$ also converges to zero exponentially, implying that $\|\xi_i\|_P^2$ decays exponentially. Hence $\gamma_i(\infty) = \gamma_i(0) + \int_0^\infty \|\xi_i(t)\|_P^2 dt$ is finite. Therefore, the adaptive gains will converge to finite constants. The limiting values depend on the initial conditions and the particular switching path, but their boundedness is guaranteed.}
\end{IEEEproof}

\begin{remark}
	So far, we have proposed a new synchronization controller. Within this approach, the dynamics of multi-agent systems can be synchronized under real-time disconnected switching topologies. Our method does not rely on the boundedness of the leader system, nor does it rely on intermittently appearing connected graphs (directed graphs containing spanning trees).
\end{remark}


\section{Further applications of network transformation mapping}\label{sec5}

Network transformation mapping given in (\ref{b1})--(\ref{b2}) can also be used for switching topology problems in other situations. This section describes two scenarios to show the important role of network transformation mapping in promoting the development of switching topology problems. 

\subsection{Multi-agent systems under jointly connected communication topologies}

We first note that the conclusions in Sections \ref{sec3} and \ref{sec4} are fundamentally predicated on Assumption 1, which characterizes a kind of real-time disconnected switching topology network. This switching paradigm is structurally distinct from the classical jointly connected topology commonly studied in multi-agent coordination.

The switching topologies characterized by Assumption \ref{assum-n} are strictly weaker than joint connectivity in one critical aspect: they impose no requirement that the union graph over any fixed interval contains a spanning tree – a fundamental requirement for jointly connected topologies. Conversely, joint connectivity offers a distinct advantage by eliminating the directed path condition. Under Assumption \ref{assum-n}, each agent must have a directed path from the leader at some time instant (not necessarily persistently), whereas joint connectivity operates without such temporal constraints.

This analysis naturally raises a pivotal question: Can multi-agent synchronization be achieved under jointly connected topologies when tracking an unbounded reference signal? We affirm that when the jointly connected topology additionally satisfies the directed path condition, our framework provides controller parameter designs guaranteeing synchronization for unbounded multi-agent systems. Crucially, we emphasize that the directed path condition is sufficient but not necessary for synchronization via network transformation mapping. Following the proof methodology of Theorems \ref{thm1}-\ref{thm2}, we formally establish this result in the following corollary.



\begin{corollary}
Considering the multi-agent systems (\ref{sys}) subjected to Assumption \ref{assum1}--\ref{assum2} and designing the control law (\ref{u2}) with adaptive coupling gain (\ref{u3}), synchronization can be achieved if the jointly connected switching topologies satisfy the directed path condition.
\end{corollary}

Compared with the case described in Assumption \ref{assum-n}, the jointly connected topology satisfying the directed path condition will naturally include some simpler processing procedures because it is a special case. Since most of the literature concerning joint connectivity assumes that a series of different topological networks appear circularly, there will be the following benefits.

1)~The parameter calculation is simpler: the coupling gain is easy to calculate even if there is no adaptive strategy. In Section \ref{sec3}, the main reason why the coupling gain is difficult to calculate is the difficulty in solving the parameter $m_k$, i.e., the length of time of $k\in\mathcal{V}_\star^\sigma$. Although we have given an estimation method before Theorem \ref{thm1}, it has highly conservative and thus inaccuracy. Therefore, the adaptive strategy is actually necessary for the case of Assumption \ref{assum-n}. However, the calculation of $m_k$ is easy in the case of jointly connected switching topologies. Since all possible communication networks have been known in advance, and they will all appear in a fixed cycle, it is easy to calculate the time proportion of $k\in\mathcal{V}_\star^\sigma$, i.e., $m_k$.

2)~The efficiency of network transformation mapping is higher. In the case of Assumption \ref{assum-n}, we cannot get the transformed network topology parameters based on the network transformation mapping in advance because it is impossible to know what the communication network will be at the next time. Therefore, a completely distributed algorithm (Algorithm \ref{alg1}) is designed to help each node calculate the network transformation mapping in real time. However, all possible communication networks can be known in advance when the topologies are jointly connected. Consequently, the transformed network parameters can be calculated in advance and called at any time during the operation of the system. Therefore, in this scenario, Algorithm \ref{alg1} is not necessary to the transformed network, but can be replaced by (\ref{b1})--(\ref{b2}).

\subsection{Distributed observer and heterogeneous multi-agent systems}

This subsection discusses the application of network transformation mapping in distributed observers targeting external systems (or leader systems). Such distributed observers serve as an effective tool for addressing cooperative control in heterogeneous multi-agent systems. Analogous to multi-agent systems, conventional approaches fail to guarantee convergence of the distributed observer's error dynamics under real-time disconnected switching topologies when the observed external system (or leader system) is unbounded. The network transformation mapping proposed in this paper resolves this fundamental challenge.

Assume an external system governed by
\begin{align}
	&\dot{x}_0=Ax_0+f(x_0),\label{sys21}\\
	&y_0=Cx_0,\label{sys22}
\end{align} 
where $x_0\in\mathbb{R}^n$, $y\in\mathbb{R}^p$ are system states and measurement outputs respectively; $f(x_0)$ is a nonlinear vector field subjected to Assumption \ref{assum1}; $A,C$ are matrices with appropriate dimension and the pair $(C,A)$ is observable. The distributed observer is an observer network including a group of local observers (agents) and a communication network. Only a part of the agents have access to the external system's outputs (or leader's output) and all the agents are supposed to reconstruct the external system's states via information exchanged with their neighbors. 

The local observer located at each agent gives rise to
\begin{align}
	&\dot{\hat{x}}_i=A\hat{x}_i+f(\hat{x}_i)+\kappa_i\mathcal{K}\zeta_i,\label{do21}\\
	&\zeta_i=\sum_{j=1}^N\beta_{ij}^\sigma\left(C\hat{x}_j-C\hat{x}_i\right)+\iota_i\left(Cx_0-C\hat{x}_i\right),\label{do22}
\end{align}
where $\hat{x}_i\in\mathbb{R}^n$ is the state estimation of $x_0$, $\kappa_i$ and $\mathcal{K}$ are coupling gain and observer gain respectively; $\beta_{ij}^\sigma$ is obtained by network transformation mapping (\ref{b1}). 

Let $e_i=\hat{x}_i-x$, and then we have
\begin{align}\label{error-do}
	\dot{e}_i&=Ae_i+f(x_i)-f(x)\notag\\
	&+\kappa_i\mathcal{K}C\left(\sum_{j=1}^N\beta_{ij}^\sigma(e_j-e_i)-\iota_ie_i\right).
\end{align}
It has a completely dual form with (\ref{error-mas}), so its stability under the switching topologies subjected to Assumption \ref{assum-n} or jointly connected switching topologies subject to directed path condition can be proved in a similar way to Theorem \ref{thm2} and Theorem \ref{thm3}. The only difference is $\mathcal{K}=P^{-1}C^TK^TP^2KC$, where $K$ is chosen so that $A-\kappa_*KC$ is Hurwitz and symmetric positive definite matrix $P$ is solved by $sym\{P(A-\kappa_*KC)\}=-\kappa_*I_n$ with $\kappa_*$ being a positive constant.

Alternatively, the local observer can be designed with adaptive coupling gain:
\begin{align}
	&\dot{\hat{x}}_i=A\hat{x}_i+f(\hat{x}_i)+(\gamma_i+\omega_i)\mathcal{K}\zeta_i,\notag\\
	&\zeta_i=\sum_{j=1}^N\beta_{ij}^\sigma\left(C\hat{x}_j-C\hat{x}_i\right)+\iota_i\left(Cx-C\hat{x}_i\right),\label{do31}\\
	&\dot{\gamma}_i=\omega_i=\left\|\zeta_i\right\|_P^2.\label{do32}
\end{align}
The stability of its error dynamics under real-time disconnected switching network can be obtained by imitating the proof of Theorem \ref{thm1}. Now, one may state the following.

\begin{corollary}\label{c2}
	Suppose the external system satisfy Assumption \ref{assum1} and the switching topologies satisfy Assumption \ref{assum-n} (or jointly connected switching topologies with directed path condition). Then, distributed observer (\ref{do21})--(\ref{do22}) (or distributed observer (\ref{do31})--(\ref{do32}) achieve $\lim_{t\to\infty}\|x-x_i\|=0$ with large enough coupling gain (or adaptive coupling gain).
\end{corollary}

With regard to this subsection and Corollary \ref{c2}, the following explanations need to be added.

\begin{figure}[!t]
	\centering
	\includegraphics[width=8cm]{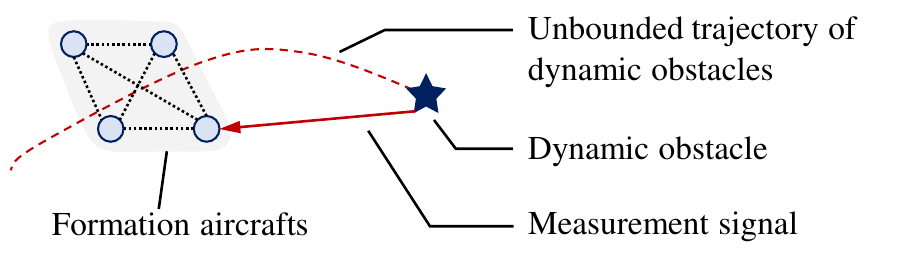}\\
	\caption{An example of the distributed observer in unbounded leader case}\label{dis-observer}
\end{figure}

Recently, a large number of articles focused on distributed observer against switching topologies \cite{10.1007/978-3-030-22808-8_3,HE2021110021,8543510}. However, similar to the multi-agent problem, the existing methods cannot solve the problem that the leader system state is unbounded and the communication topology does not contain a spanning tree at any time. The issue of the distributed observer in the case of an unbounded leader state is not rare. For example, Fig. \ref{dis-observer} shows a formation avoiding dynamic obstacles. The agents in the formation need to prevent a dynamic obstacle, while only some agents can measure the obstacle information. Other agents need to estimate the obstacle dynamics through the distributed observer network and then take avoidance actions based on their estimation. In this scenario, it is obvious that the trajectory of dynamic obstacles cannot be assumed to be bounded. Therefore, Corollary \ref{c2} is very valuable when encountering similar problems. It also reflects the value of the network transformation mapping proposed in this paper.

Finally, we would like to discuss the heterogeneous multi-agent systems based on distributed observer. In this scenario, each follower reconstructs the state of the leader system through the distributed observer, and then a pure decentralized control law for keeping synchronization can be established based on the state estimation. Therefore, the distributed-observer-based distributed control law can solve the related problems of heterogeneous multi-agent systems. It is equivalent to solving the synchronization  problem of heterogeneous multi-agent systems by using the homogeneity multi-agent method. The network transformation mapping proposed in this paper has proved that it can effectively deal with homogeneous multi-agent systems and distributed observers. Consequently, it can also be applied to heterogeneous multi-agent systems through distributed observers. This shows that the method proposed in this paper is a general method when the reference system (leader) is unbounded, and the communication network does not contain a spanning tree at any time. It can play a significant role whether facing homogeneous multi-agent systems, distributed observer systems, or heterogeneous multi-agent systems. In the future, it is also expected to play a decisive role when dealing with switching topologies in more fields.
\begin{figure*}[!t]
	\centering
	\includegraphics[width=16cm]{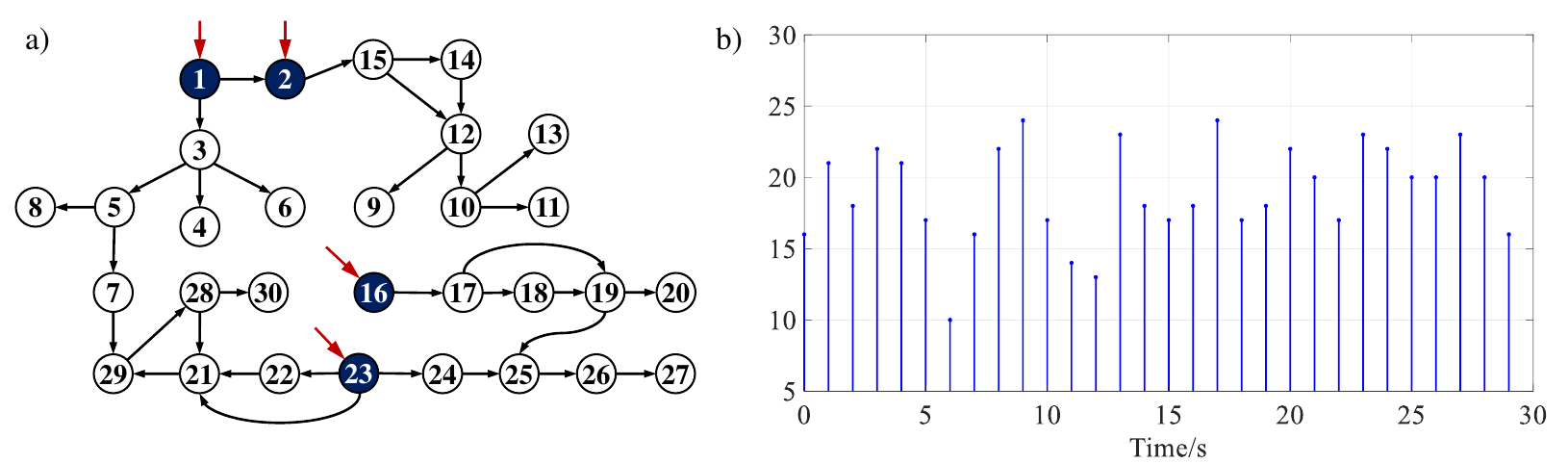}\\
	\caption{Communication networks among all agents. a) The topology when all links work normal, and all arcs, whether red or black, have a probability of interruption; b) Number of links that have no access to virtual leader}\label{network}
\end{figure*}

\begin{figure*}[!t]
	\centering
	\includegraphics[width=16cm]{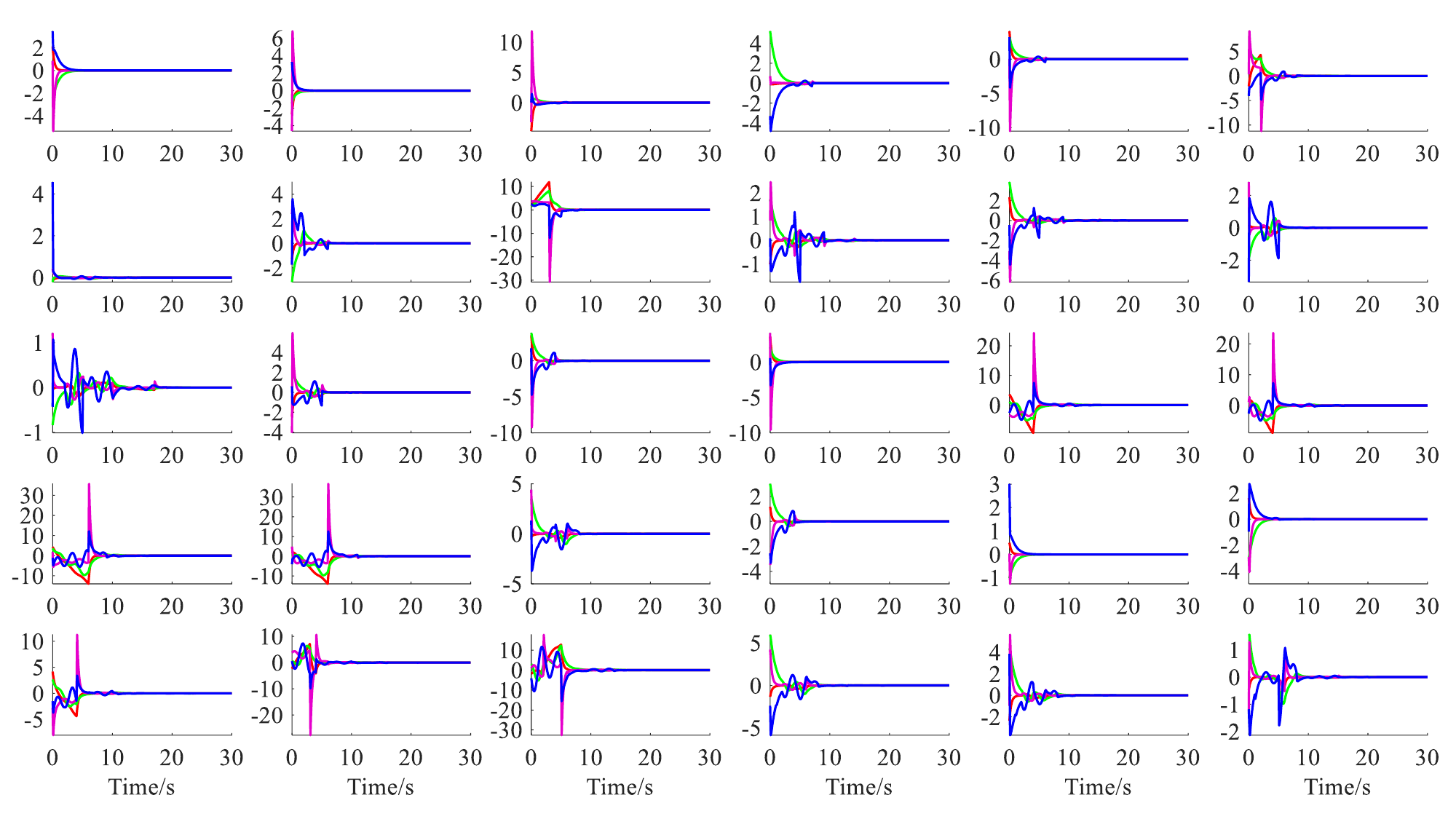}\\
	\caption{Time varying trajectories of error dynamics for each agent}\label{error-dynamics}
\end{figure*}

\begin{figure}[!t]
	\centering
	\includegraphics[width=8cm]{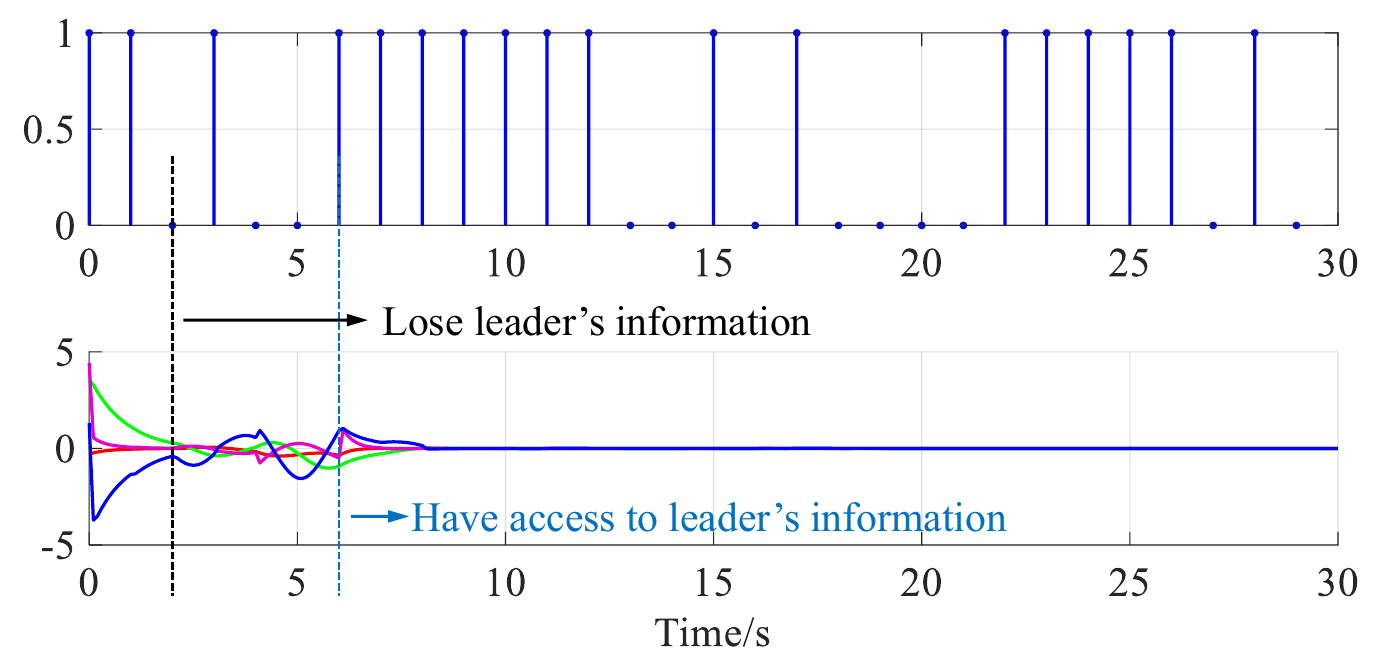}\\
	\caption{Relationship between dynamic performance and the structure communication topology. {\color{blue}In the subfigure, at time $t$, the vertical axis value is $1$ if the agent is a descendant of the leader during $[t, t+1)$, and $0$ otherwise.}}\label{follower-states}
\end{figure}

\begin{figure}[!t]
	\centering
	\includegraphics[width=8cm]{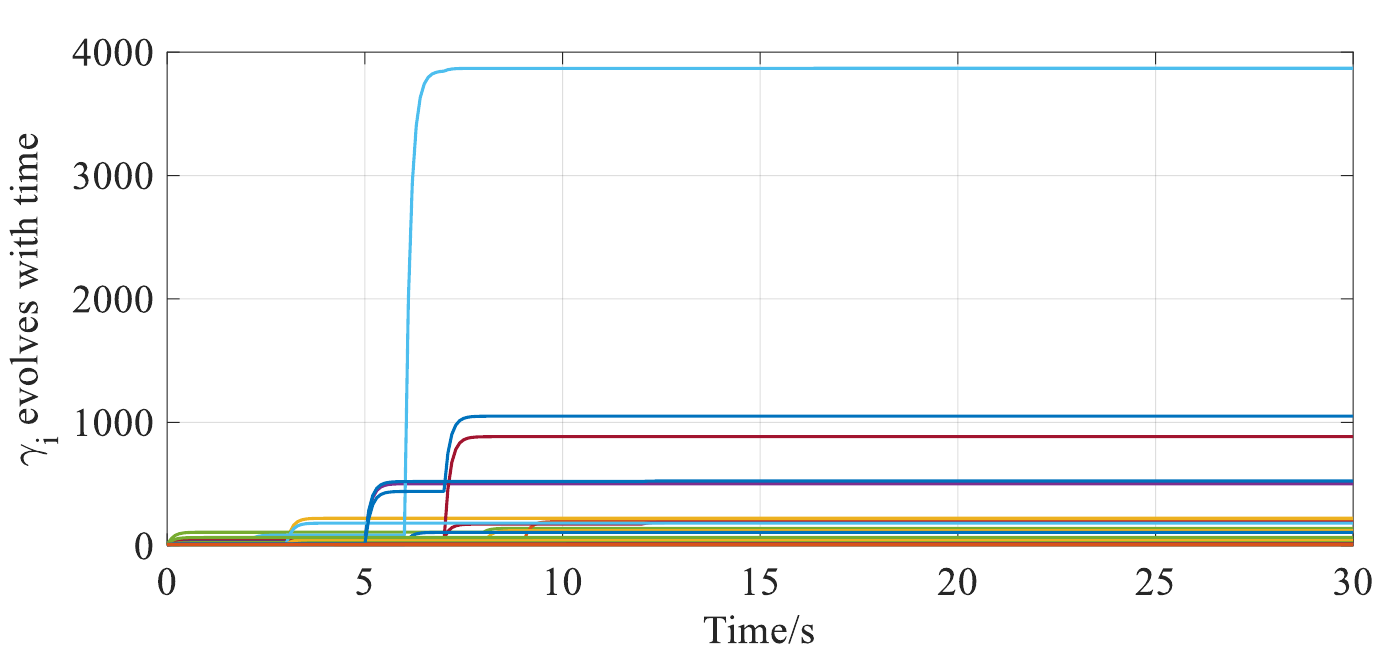}\\
	\caption{Time varying trajectory of coupling gain}\label{coupling}
\end{figure}



\section{Simulation}\label{sec6}

This section considers two simulations. The multi-agent systems composed of 30 one-link manipulators are employed to show the effectiveness of the methods developed by Theorem \ref{thm2}--\ref{thm1}. Then, Section \ref{sec5.2} studies a multi-motor synchronous control system based on Fieldbus to display the function of network transformation mapping when the communication network obeys joint connectivity property. Since the distributed observer is a dual problem of multi-agent, it is no longer simulated for Corollary \ref{c2}.

\subsection{Multi-agent systems of one-link manipulators}\label{sec5.1}
The $i$th one-link manipulator with the elastic shaft is governed by the dynamics
\begin{align}
	\dot{x}_i=Ax_i+f(x_i)+Bu_i,~i=1,\ldots,30,\label{sys-simu}
\end{align}
where
\begin{align}
	&A=\begin{bmatrix}
		 0     &    0  &  1 &        0\\
		0       &  0   &      0  &  1\\
		-1.13 &   1.33 &  -0.10  &   0\\
		3.33  & -5   &      0  & 0.28
	\end{bmatrix},\\
	&B=\begin{bmatrix}
		0&1&0&0\\0&0&0&1
	\end{bmatrix}^T,\\
	&f(x_i)=\begin{bmatrix}
		0&0&1/5&\frac{1}{2}\cos(\frac{x_{i2}}{2})
	\end{bmatrix}^T.
\end{align}
Noticed that there exists a constant external torque $1/5$ in nonlinear term $f(x_i)$, which makes the reference system $\dot{x}_0=Ax_0+f(x_0)$ be the open-loop unbounded system. Therefore, systems (\ref{sys-simu}) are more conducive to showing the contribution of this paper. Specifically, we will show the synchronization of (\ref{sys-simu}) against the switching topologies without relying on a specific switching mode or a special graph (such as a graph containing a spanning tree).

To this end, the completely healthy communication network between all agents is shown in Fig. \ref{network}a), where agents $1,2,16,23$ are pinning nodes. At the initial time, we randomly disconnect $10$ links in the graph to ensure that there is no directed spanning tree at the initial time. Then, each broken link will be repaired after a random time interval obeying negative exponential distribution with parameter $1.5$. In addition, each normal link may fail after a random time interval which obeys negative exponential distribution with parameter $1$.

Fig. \ref{network}b) shows the number changing over time of agents that cannot directly or briefly obtain the virtual leader information. As can be seen from the figure, in most of the time, most nodes have no access to virtual leader information directly or indirectly.
Therefore, it is a challenging simulation environment because the cited multi-agent systems are open-loop unbounded, and the real-time disconnected communication topology evolves without any periodicity or subjectivity (fast switching). 

By designing distributed control law from (\ref{u1}) and (\ref{u2}) as well as the network transformation (\ref{b1})--(\ref{b2}), the dynamics of all 30 agents are shown in Fig. \ref{error-dynamics}, which indicates the availability of our method.

In order to explain the internal operation mechanism of our method in more detail, we draw Fig. \ref{follower-states} to show the variation of error dynamics with the change in the communication network. It is the details of the 21st agent located at row $4$ column $3$ in Fig. \ref{error-dynamics}. In Fig. \ref{follower-states}, the first subfigure explains when this agent has access to the leader's information. By comparing the first subfigure with the second subfigure, one may notice the error dynamics have a convergence trend at the initial time owing to its access to the leader's information. Then, divergence occurs at $t=3$ since the path between the leader and the $21$st agent is broken. In this time interval, the dynamics of the $21$st agent are influenced by other agents that have not been able to achieve synchronization. Finally, error dynamics converge again at $t=6$ because the path is repaired. This illustrates the necessity of designing the convergence rate in Theorem 2.

Finally, Fig. \ref{coupling} displays the time-varying trajectories of coupling gains $\gamma_i$ for $i=1,\ldots,N$. This result verifies Theorem \ref{thm3} and illustrates the effectiveness of adaptive coupling gain, which means the amount of calculation in the process of parameter design is greatly reduced.
\begin{figure*}[!t]
	\centering
	\includegraphics[width=0.9\textwidth]{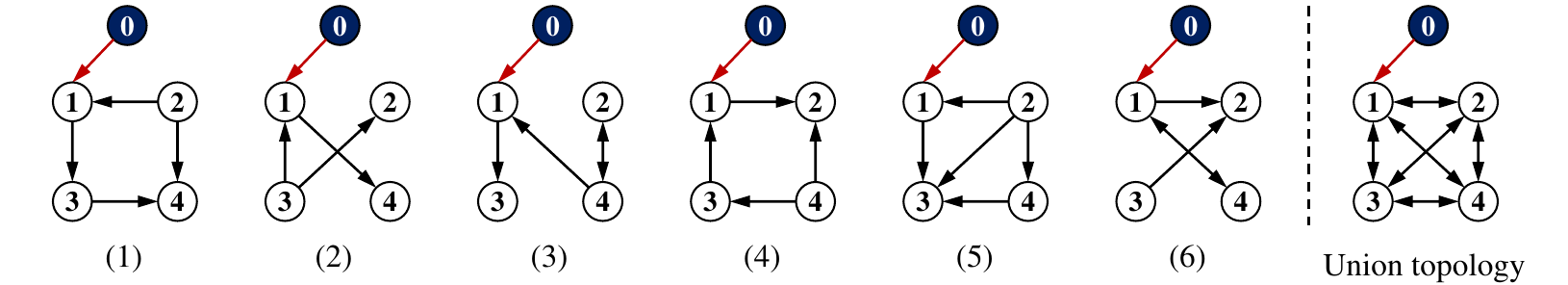}\\
	\caption{Communication networks among four agents}\label{graph}
\end{figure*}
\begin{figure*}[!t]
	\centering
	\includegraphics[width=0.9\textwidth]{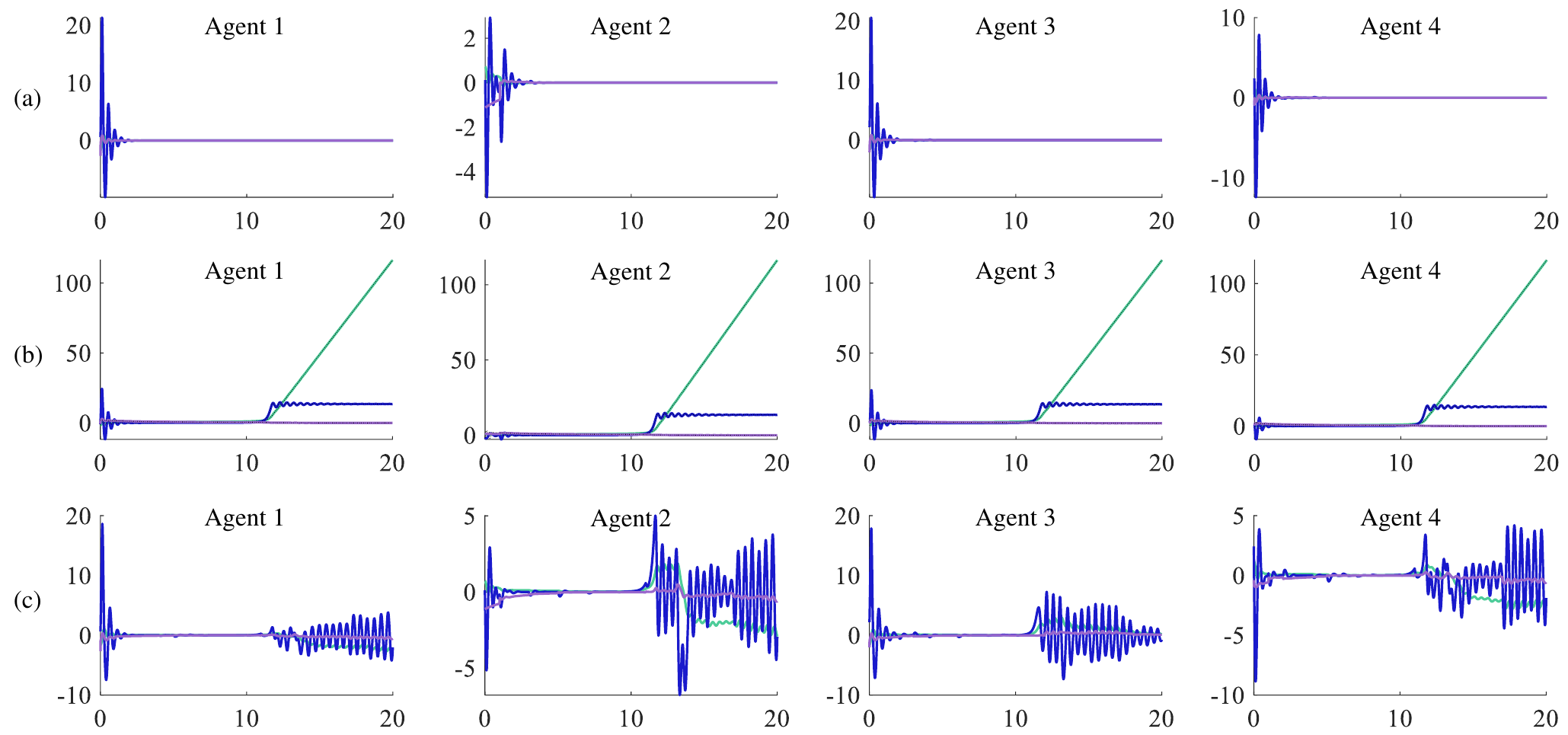}\\
	\caption{Time varying trajectories of error dynamics for each agent. Line a) shows the tracking errors of the control protocol with network transformation. And Line b) is the associated  trajectories of four agents. Line c) shows the error dynamics of control protocol without network transformation.}\label{error-dynamics-com}
\end{figure*}

\subsection{Multi-motor synchronous control systems based on Fieldbus}\label{sec5.2}

A multi-motor synchronous control system connected with Fieldbus gives rise to
\begin{align}
	&M\ddot{\delta}_i=P_W-D\dot{\delta}_i-\eta_1 E_q\sin{\delta_i}\notag\\
	&\varrho\dot{E}_{i,q}=-\eta_2E_q+\eta_3\cos{\delta_i}+E_{i,FD}.
\end{align}
In this system, $\delta_i$ stands for angle (rad); $E_{i,q}$ is voltage; $P_W$ is the mechanical input power; $D$ and $M$ represent damping coefficient and inertia coefficient respectively; $\eta_1$, $\eta_2$, and $\eta_3$ are constants; and $E_{i,FD}$, the field voltage, is regarded as the control input. By letting $x_{i,1}=\delta_i$, $x_{i,2}=\dot{\delta}_i$, $x_{i,3}=E_{i,q}$, and $u_i=E_{i,FD}$, we have
\begin{align}
	\dot{x}_i=Ax_i+f(x_i)+Bu_i,
\end{align} 
where $x_i=col\{x_{i,1},x_{i,2},x_{i,3}\}$, and
\begin{align}
	&A=\begin{bmatrix}0&1&0\\0&\frac{D}{M}&\frac{\eta_1}{M}\\0&0& -\frac{\eta_2}{\varrho}\end{bmatrix},~B=\begin{bmatrix}0\\0\\ \frac{1}{\varrho}\end{bmatrix},\notag\\
	&f(x_i)=\begin{bmatrix}0\\ \frac{P_W}{M}-\frac{\eta_1}{M}x_{i,3}-\frac{\eta_1}{M}x_{i,3}\sin{x_{i,1}}\\ \frac{\eta_3}{\varrho}\cos{x_{i,1}}\end{bmatrix}.
\end{align}
The leader system is governed by
\begin{align}
	\dot{x}_0=Ax_0+f(x_0).
\end{align} 

To facilitate comparison, this section considers $4$ followers, and the communication networks among them are shown in Fig. \ref{graph}. It is assumed that $6$ networks in this figure appear in turn with a time interval of $1$. The parameters of this simulation are set as $P_W=0.815$, $M=0.0147$, $D=4M$, $\varrho=6.6$, and $\eta_1=2$, $\eta_2=2.7$, $\eta_3=1.7$.

Fig.~\ref{error-dynamics-com} presents the simulation results. Line a) displays the error dynamics and line b) the state trajectories under the control protocol with network transformation. The error dynamics in line a) converge to zero, while the trajectories in line b) are unbounded. The subfigures in line b) indicate that the reference signal is bounded for the first $11$ seconds and becomes unbounded thereafter. This property helps illustrate the role of network transformation mapping. As shown in line c) (it shows the error dynamics of control protocol without network transformation), synchronization is achieved for $t<11$, but the error dynamics tend to diverge for $t>11$. This is because, after $11$ seconds, nodes without access to the leader's information interfere with those that do have access when the reference signal becomes unbounded, leading to synchronization failure. By contrast, the control protocol with network transformation mapping (line a)) successfully overcomes this problem. 



\section{Conclusions}\label{sec7}

This paper has addressed the problem of multi-agent synchronization under switching topologies that never contain a spanning tree, with an unbounded reference signal. Communication links have been assumed subject to stochastic failures/attacks modeled by negative exponential distributions. In this setting, conventional methods such as fast switching and intermittent communication have become no longer applicable. To overcome this, we have proposed a novel synchronization controller based on network transformation mapping, which achieves synchronization with an unbounded reference signal. Adaptive coupling gain has been introduced to avoid complex gain calculations. Finally, two numerical examples—multi‑one‑link manipulator systems and multi‑motor synchronous control systems based on Fieldbus—have verified the effectiveness of the developed strategy.

\appendices \section{Proof of Lemma \ref{trans}}\label{app-L3}
\begin{IEEEproof}
		\color{blue}We first recall the definition of $\ell_{i}^\sigma$: $\ell_{ik}^\sigma = 1$ if and only if there exists a path from some pinning node $k_j \in \mathcal{I}^\sigma$ to node $i$; otherwise $\ell_{i}^\sigma = 0$. Based on this, we partition the node set $\mathcal{V}$ into two separated subsets $\mathcal{V}_{\star}^\sigma = \{ i \mid \ell_{i}^\sigma = 1 \}$ and $\mathcal{V}_{\diamond}^\sigma = \{ i \mid \ell_{i}^\sigma = 0 \}$.	By construction, $\mathcal{V}_{\star}^\sigma$ is exactly the set of all pinning nodes and all nodes that can reach some pinning node (i.e., pinning nodes and their descendants), which is precisely the node set of $\mathcal{G}_\star^\sigma$. The complement $\mathcal{V}_{\diamond}^\sigma$ is the node set of $\mathcal{G}_\diamond^\sigma$.
		
		Now, we examine whether the transformed adjacency weights $\beta_{ij}^\sigma = \alpha_{ij}^\sigma \delta_{ij}^\sigma$ meet the requirement of Lemma \ref{trans}. To this end, we consider all possible cases for the pair $(i,j)$.
		
		\textbf{Case 1:} Both $i$ and $j$ are in $\mathcal{V}_{\star}^\sigma$. For any $i,j \in \mathcal{V}_{\star}^\sigma$, we have $\ell_{i}^\sigma = 1$ and $\ell_{j}^\sigma = 1$ by definition. Hence, based on (\ref{b2}), we obtain $\delta_{ij}^\sigma = 1 \wedge 1 = 1$, so $\beta_{ij}^\sigma = \alpha_{ij}^\sigma$.
		Thus every arc that originally existed from $j$ to $i$ (i.e., $\alpha_{ij}^\sigma = 1$) remains unchanged; arcs that did not exist remain absent.
		
		\textbf{Case 2:} Both $i$ and $j$ are in $\mathcal{V}_{\diamond}^\sigma$.
		For $i,j \in \mathcal{V}_{\diamond}^\sigma$, we have $\ell_{i}^\sigma = 0$ and $\ell_{j}^\sigma = 0$, so $\delta_{ij}^\sigma = 0 \wedge 0 = 0$ and hence $\beta_{ij}^\sigma = 0$.
		It means all arcs inside $\mathcal{V}_{\diamond}^\sigma$ are removed, regardless of the original $\alpha_{ij}^\sigma$.
		
		\textbf{Case 3:} One node is in $\mathcal{V}_{\star}^\sigma$ and the other in $\mathcal{V}_{\diamond}^\sigma$. 
		Without loss of generality, suppose $i \in \mathcal{V}_{\star}^\sigma$ and $j \in \mathcal{V}_{\diamond}^\sigma$. Then $\ell_{i}^\sigma = 1$, $\ell_{j}^\sigma = 0$ giving $\delta_{ij}^\sigma = 1 \wedge 0 = 0$, which leads to $\beta_{ij}^\sigma = 0$. Hence, all arcs from $\mathcal{V}_{\star}^\sigma$ to $\mathcal{V}_{\diamond}^\sigma$ and from $\mathcal{V}_{\diamond}^\sigma$ to $\mathcal{V}_{\star}^\sigma$ are eliminated.
		
		Consequently, after the transformation: arcs inside $\mathcal{G}_\star^\sigma$ are preserved exactly as in the original graph; arcs inside $\mathcal{G}_\diamond^\sigma$ are completely removed; and all arcs between $\mathcal{G}_\star^\sigma$ and $\mathcal{G}_\diamond^\sigma$ (in either direction) are also completely removed. This completes the proof.
\end{IEEEproof}

\section{Proof of Theorem \ref{thm2}}\label{app-Th1}
\begin{IEEEproof}
		Denote $e_i=x_i-x_0$ as the tracking error of the $i$th node. Then, two categories of graph leads to the compact form $e_\star=col\{e_i,~i\in\mathcal{V}_\star^\sigma\}$, and $e_\diamond=col\{e_i,~i\in\mathcal{V}_\diamond^\sigma\}$. Let $\xi_{i}=\sum_{j=1}^N\beta_{ij}^\sigma\left(x_j-x_i\right)+\iota_i^\sigma\left(x_0-x_i\right)$, hence one deduces $\xi_\star=(\mathcal{H}_\star^\sigma\otimes I_n)e_\star$ and $\xi_\star=col\{\xi_i,~i\in\mathcal{V}_\star^\sigma\}$. The dynamics of $e_i$ give rise to
\begin{align}\label{error-mas}
	\dot{e}_i=&Ae_i+f(x_i)-f(x_0)\notag\\
	&+\kappa_iB\mathcal{K}\left(\sum_{j=1}^N\beta_{ij}^\sigma(x_j-x_i)+\iota_i^\sigma(x_0-x_i)\right).
\end{align}
Therefore, one can also deduce
\begin{align}
	\dot{\xi}_\star=&(\mathcal{H}_\star^\sigma\otimes I_n)\dot{e}_\star\notag\\
	=&(\mathcal{H}_\star^\sigma\otimes A)e_\star+(\mathcal{H}_\star^\sigma\otimes I_n)\tilde{f}_\star(\bar{x}_\star,\bar{x}_{0,\star})\notag\\
	&-\left(\mathcal{H}_\star^\sigma\Upsilon_\star^\sigma\mathcal{H}_\star^\sigma\otimes B\mathcal{K}\right)e_\star\notag\\
	=&(I_{\pi(\sigma)}\otimes A)\xi_\star+(\mathcal{H}_\star^\sigma\otimes I_n)\tilde{f}_\star(\bar{x}_\star,\bar{x}_{0,\star})\notag\\
	&-\left(\mathcal{H}_\star^\sigma\Upsilon_\star^\sigma\otimes B\mathcal{K}\right)\xi_\star,
\end{align}
where the last equal sign holds because $\mathcal{H}_\star^\sigma$ is a nonsingular matrix; $\tilde{f}_\star(\bar{x}_\star,\bar{x}_{0,\star})=col\{\tilde{f}_i(x_i,x_0),~i\in\mathcal{V}_\star^\sigma\}$ with $\tilde{f}_i(x_i,x_0)=f(x_i)-f(x_0)$, $\bar{x}_\star=col\{x_i,~i\in\mathcal{V}_\star\}$, and $\bar{x}_{0,\star}=1_{\pi(\sigma)}\otimes x_0$; $\pi(\sigma)$ represents the cardinal number of $\mathcal{V}_\star^\sigma$; $\Upsilon_\star^\sigma=diag\{\kappa_i,~i\in\mathcal{V}_\star^\sigma\}$. 

Set $\mathcal{K}=KK^TB^TP$ and choose a positive definite matrix $\Theta=diag\{\theta_i^\sigma,~i\in\mathcal{V}_\star^\sigma\}$ satisfying $sym\{\Theta_\star^\sigma\mathcal{H}_\star^\sigma\}>\kappa_*I_{\pi(\sigma)}$. Then, the Lyapunov candidate can be chosen as $V_1=\xi_\star^T(\Upsilon_\star^\sigma\Theta_\star^\sigma\otimes P)\xi_\star$, and its derivative along with $\xi_\star$ yields
\begin{align}\label{V1}
	\dot{V}_1(t)=&\xi_\star^T\left(\Upsilon_\star^\sigma\Theta_\star^\sigma\otimes(PA+A^TP)\right)\xi_\star\notag\\
	&-\xi_\star^T\left(\Upsilon_\star^\sigma sym\{\Theta_\star^\sigma\mathcal{H}_\star^\sigma\}\Upsilon_\star^\sigma\otimes PBKK^TB^TP\right)\xi_\star\notag\\
	&+2\xi_\star^T(\Upsilon_\star^\sigma\Theta_\star^\sigma\mathcal{H}_\star^\sigma\otimes P)\tilde{f}_\star(\bar{x}_\star,\bar{x}_{0,\star})\notag\\
	\leq&\xi_\star^T\left(\Upsilon_\star^\sigma\Theta_\star^\sigma\otimes(PA+A^TP)\right)\xi_\star\notag\\
	&-\kappa_*\xi_\star^T\left(\Upsilon_\star^\sigma\Upsilon_\star^\sigma\otimes PBKK^TB^TP\right)\xi_\star\notag\\
	&-\kappa_*\xi_\star^T\left(\Theta_\star^\sigma\Theta_\star^\sigma\otimes I_n\right)\xi_\star+\kappa_*\xi_\star^T\left(\Theta_\star^\sigma\Theta_\star^\sigma\otimes I_n\right)\xi_\star\notag\\
	&+2\xi_\star^T(\Upsilon_\star^\sigma\Theta_\star^\sigma\mathcal{H}_\star^\sigma\otimes P)\tilde{f}_\star(\bar{x}_\star,\bar{x}_{0,\star}).
\end{align}

Note that
\begin{align}
	&\kappa_*\xi_\star^T\left(\left(\Upsilon_\star^\sigma\right)^2\otimes PBKK^TB^TP\right)\xi_\star+\kappa_*\xi_\star^T(\Theta_\star^\sigma\Theta_\star^\sigma\otimes I_n)\xi_\star\notag\\
	=&\kappa_*\sum_{i\in\mathcal{V}_\star^\sigma}\xi_i^T\left(\kappa_i^2PBKK^TB^TP+(\theta_i^\sigma)^2I_n\right)\xi_i\notag\\
	=&\kappa_*\sum_{i\in\mathcal{V}_\star^\sigma}\xi_i^T\left(\left(\kappa_iPBK\right)\left(\kappa_iPBK\right)^T+(\theta_i^\sigma)^2I_n\right)\xi_i\notag\\
	\geq&\kappa_*\sum_{i\in\mathcal{V}_\star^\sigma}\xi_i^T\left(\theta_i^\sigma\kappa_iPBK+\theta_i^\sigma\kappa_iK^TB^TP\right)\xi_i\notag\\
	=&\kappa_*\xi_\star^T\left(\Upsilon_\star^\sigma\Theta_\star^\sigma\otimes\left(PBK+K^TB^TP\right)\right)\xi_\star.
\end{align}

Subsequently, find $K$ such that $A-\kappa_*BK$ is Hurwitz and further find symmetric positive definite matrix $P$ so that $sym\{P(A-\kappa_*BK)\}=-\kappa_*I_n$. Then,
\begin{align}\label{t1dv1}
	&\xi_\star^T(\Upsilon_\star^\sigma\Theta_\star^\sigma\otimes (PA+A^TP))\xi_\star\notag\\
	-&\kappa_*\xi_\star^T\left(\Upsilon_\star^\sigma \Upsilon_\star^\sigma\otimes PBKK^TB^TP\right)\xi_\star\notag\\
	-&\kappa_*\xi_\star^T\left(\Theta_\star^\sigma\Theta_\star^\sigma\otimes I_n\right)\xi_\star+\kappa_*\xi_\star^T\left(\Theta_\star^\sigma\Theta_\star^\sigma\otimes I_n\right)\xi_\star\notag\\
	\leq&\xi_\star^T(\Upsilon_\star^\sigma\Theta_\star^\sigma\otimes (PA+A^TP))\xi_\star\notag\\
	-&\kappa_*\xi_\star^T\left(\Upsilon_\star^\sigma\Theta_\star^\sigma\otimes\left(PBK+K^TB^TP\right)\right)\xi_\star\notag\\
	+&\kappa_*\xi_\star^T\left(\Theta_\star^\sigma\Theta_\star^\sigma\otimes I_n\right)\xi_\star\notag\\
	=&\xi_\star^T(\Upsilon_\star^\sigma\Theta_\star^\sigma\otimes sym\{P(A-\kappa_*BK)\})\xi_\star\notag\\
	+&\kappa_*\xi_\star^T\left(\Theta_\star^\sigma\Theta_\star^\sigma\otimes I_n\right)\xi_\star\notag\\
	\leq&\left(-\kappa_*\underline{\kappa}_\star^\sigma\underline{\theta}_\star^\sigma+\kappa_*\left(\bar{\theta}_\star^\sigma\right)^2\right)\|\xi_\star\|^2.
\end{align}

Now we focus on the nonlinear term $\xi_\star^T(\Upsilon_\star^\sigma\Theta_\star^\sigma\mathcal{H}_\star^\sigma\otimes P)\tilde{f}_\star(\bar{x}_\star,\bar{x}_{0,\star})$. Denote $\wp_{\star1}=\rho\bar{\kappa}_\star^\sigma\bar{\theta}_\star^\sigma\bar{\lambda}(P)\bar{\lambda}(\mathcal{H}_\star^\sigma)\underline{\lambda}(\mathcal{H}_\star^\sigma)$, and thus yield
\begin{align}\label{t1dv2}
	&\xi_\star^T(\Upsilon_\star^\sigma\Theta_\star^\sigma\mathcal{H}_\star^\sigma\otimes P)\tilde{f}_\star(\bar{x}_\star,\bar{x}_{0,\star})\notag\\
	\leq&\frac{\wp_{\star1}}{\rho\underline{\lambda}(\mathcal{H}_\star^\sigma)}\|\xi_\star\|\|\tilde{f}_\star(\bar{x}_\star,\bar{x}_{0,\star})\|\notag\\
	=&\frac{\wp_{\star1}}{\rho\underline{\lambda}(\mathcal{H}_\star^\sigma)}\|\xi_\star\|\sum_{i\in\mathcal{V}_\star^\sigma}\|\tilde{f}_i(x_i,x_0)\|\notag\\
	=&\frac{\wp_{\star1}}{\underline{\lambda}(\mathcal{H}_\star^\sigma)}\|\xi_\star\|\sum_{i\in\mathcal{V}_\star^\sigma}\|e_i\|\leq\frac{\wp_{\star1}}{\underline{\lambda}(\mathcal{H}_\star^\sigma)}\|\xi_\star\|\left(\sum_{i\in\mathcal{V}_\star^\sigma}\|e_i\|\right)^{2\times \frac{1}{2}}\notag\\
	\leq&\frac{\wp_{\star1}}{\underline{\lambda}(\mathcal{H}_\star^\sigma)}\|\xi_\star\|\left(N\sum_{i\in\mathcal{V}_\star^\sigma}\|e_i\|^2\right)^{\frac{1}{2}}=\frac{\sqrt{N}\wp_{\star1}}{\underline{\lambda}(\mathcal{H}_\star^\sigma)}\|\xi_\star\|\|e_\star\|\notag\\
	\leq&\sqrt{N}\wp_{\star1}\|\xi_\star\|^2.
\end{align}

Substituting (\ref{t1dv1}) and (\ref{t1dv2}) into (\ref{V1}), one has
\begin{align}\label{V2}
	\dot{V}_1(t)\leq&\left(-\kappa_*\underline{\kappa}_\star^\sigma\underline{\theta}_\star^\sigma+\kappa_*\left(\bar{\theta}_\star^\sigma\right)^2\right)\|\xi_\star\|^2\notag\\
	+&2\sqrt{N}\rho\bar{\lambda}(P)\bar{\lambda}(\mathcal{H}_\star^\sigma)\underline{\lambda}(\mathcal{H}_\star^\sigma)\bar{\kappa}_\star^\sigma\bar{\theta}_\star^\sigma\|\xi_\star\|^2.
\end{align}
Choose $\kappa_*=(\bar{\kappa}_\star^\sigma)^{1/2}$ such that $\dot{V}_1<0$ for a group of large enough $\kappa_i$. Then, substituting the relationship (\ref{rela}) into (\ref{V2}), one has
\begin{align}
	\dot{V}_1\leq -\frac{\bar{\kappa}_\star^\sigma\bar{\theta}_\star^\sigma}{2}\wp_\star\|\xi_\star\|^2\leq-2\wp_\star V_1.
\end{align}
Therefore, $V_1(t)\leq V_1(t_m)\exp\{-2\wp_\star (t-t_m)\}$, where $t_m=\max_k\{t_k\leq t\}$ with $t_k$ being the switching instant. Therefore,
\begin{align}
	&\frac{1}{2}\underline{\lambda}(P)\underline{\kappa}_\star^\sigma\underline{\theta}_\star^\sigma\|\xi_\star\|^2\leq\frac{1}{2}\xi_\star^T\left(\Upsilon_\star^\sigma\Theta_\star^\sigma\otimes P\right)\xi_\star\notag\\
	\leq&\frac{1}{2}\xi_\star^T(t_m)\left(\Upsilon_\star^\sigma\Theta_\star^\sigma\otimes P\right)\xi_\star(t_m)\exp\{-2\wp_\star (t-t_m)\}\notag\\
	\leq&\frac{1}{2}\bar{\kappa}_\star^\sigma\bar{\theta}_\star^\sigma\bar{\lambda}(P)\|\xi_\star(t_m)\|^2\exp\{-2\wp_\star (t-t_m)\},
\end{align}
where $\underline{\theta}^\sigma=\min\{\theta_i^\sigma,~i\in\mathcal{V}_\star^\sigma\}$. We thus have
\begin{align}
	\|\xi_\star(t)\|\leq\sqrt{\frac{\bar{\kappa}_\star^\sigma\bar{\theta}_\star^\sigma\bar{\lambda}(P)}{\underline{\kappa}_\star^\sigma\underline{\theta}_\star^\sigma\underline{\lambda}(P)}}\|\xi_\star(t_m)\|\exp\{-\wp_\star (t-t_m)\}.
\end{align}
and
\begin{align}
	\|e_i(t)\|\leq\sqrt{\frac{\bar{\kappa}_\star^\sigma\bar{\theta}_\star^\sigma\bar{\lambda}(P)\bar{\lambda}(\mathcal{H}_\star^\sigma)}{\underline{\kappa}_\star^\sigma\underline{\theta}_\star^\sigma\underline{\lambda}(P)\underline{\lambda}(\mathcal{H}_\star^\sigma)}}\|e_i(t_m)\|\exp\{-\wp_\star (t-t_m)\}.
\end{align}
It shows that the convergence rate of error dynamics is at least $\wp_\star$, which leads to the validity of (\ref{rela}).
\end{IEEEproof}



%





\ifCLASSOPTIONcaptionsoff
  \newpage
\fi





\bibliographystyle{IEEEtran}
\bibliography{trans}
%





\vfill


\end{document}